\documentclass[11pt,english,draft]{article}
\usepackage{amsmath}
\usepackage{amssymb}
\usepackage{amsthm}
\usepackage[english]{babel}
\usepackage{times}
\usepackage[T1]{fontenc}
\usepackage[all]{xy}
\usepackage{geometry}

\theoremstyle{plain}
\newtheorem{teo}{Theorem}[section]
\newtheorem{coro}[teo]{Corollary}
\newtheorem{lem}[teo]{Lemma}
\newtheorem{pro}[teo]{Proposition}

\newtheorem{defn}[teo]{Definition}

\theoremstyle{definition}
\newtheorem{oss}[teo]{Remark}

\newcommand{\B}{\mathbb{B}}
\newcommand{\HH}{\mathbb{H}}

\newcommand{\C}{\mathbb{C}}
\newcommand{\rr}{\mathbb{R}}
\newcommand{\s}{\mathbb{S}}

\newcommand{\Hi}{\mathcal{H}}

\newcommand{\de}{\partial_c}
\newcommand{\p}{\partial}

\DeclareMathOperator{\IIm}{Im}
\DeclareMathOperator{\length}{length}
\DeclareMathOperator{\RRe}{Re}

\newcommand{\III}{\mathcal{I}}
\newcommand{\RR}{\mathbb{R}}
\newcommand{\BB}{\mathbb{B}}

\newcommand{\II}{\mathbb{I}}
\renewcommand{\SS}{\mathbb{S}}

\title{\bf Invariant metrics for the quaternionic Hardy space}
\author{Nicola Arcozzi\thanks{Partially supported by the PRIN project Real and Complex Manifolds of the Italian MIUR and by INDAM-GNAMPA},
Giulia Sarfatti\thanks{Partially supported by INDAM-GNSAGA and by the PRIN project Real and Complex Manifolds of the Italian MIUR}}

\date{}

\begin{document}

\maketitle

\begin{abstract}
We find Riemannian metrics on the unit ball of the quaternions, which are naturally associated with reproducing kernel Hilbert spaces. We study the metric arising 
from the Hardy space in detail. 
We show that, 
in contrast to the one-complex variable case, no Riemannian metric is invariant under regular self-maps of the quaternionic ball.
\end{abstract}

\bigskip


{
\small
\noindent \bf Notation. \it The symbol $\HH$ denotes the set of the quaternions $q=x_0+x_1i+x_2j+x_3k=\RRe(q)+\IIm(q)$, with $\RRe(q)=x_0$ and $\IIm(q)=x_1i+x_2j+x_3k$;  
where the $x_j$'s are real numbers and the imaginary units
$i,j,k$ are subject to the rules $ij=k,\ jk=i,\ ki=j$ and $i^2=j^2=k^2=-1$. We identify the quaternions $q$ whose imaginary part vanishes, $\IIm(q)=0$, 
with real numbers, $\RRe(q)\in\RR$; 
and, similarly, we let $\II=\RR i+\RR j+\RR k$ be the set of the imaginary quaternions. The norm $|q|\ge0$ of $q$ is 
$|q|=\sqrt{\sum_{l=0}^3 x_l^2}=(q\overline{q})^{1/2}$, where $\overline{q}=x_0-x_1i-x_2j-x_3k$ is the conjugate of $q$. 
The open unit ball $\BB$ in $\HH$ contains the quaternions $q$ such that $|q|<1$. The boundary of $\BB$ in $\HH$
is denoted by $\partial\BB$. By the symbol $\SS$ we denote the unit sphere of the imaginary quaternions: $q\in\II$ belongs to $\SS$ if $|q|=1$. For $I$ in $\s$, the slice 
$L_I=L_{-I}$ in $\HH$ contains all quaternions having the form $q=x+yI$, with $x,y$ in $\RR$.
If $f$ is a real differentiable function on a domain $\Omega \subseteq \HH$, we denote its real differential at a point $w\in\Omega$ by the symbol $f_*[w]$. 
\rm
}

\section{Introduction}

Let $\HH$ be the skew-field of the quaternions.
The quaternionic Hardy space $H^2(\BB)$ consists of the formal power series of the quaternionic variable $q$,
\begin{equation*}
 f(q)=\sum_{n=0}^\infty q^na_n,
\end{equation*}
such that the sequence of quaternions $\{a_n\}$ satisfies
\begin{equation}\label{htwo}
\|f\|_{H^2(\BB)}:= \left(\sum_{n=0}^\infty \left|a_n\right|^2\right)^{1/2}<\infty.
\end{equation}
It is easily verified that the series converges to a function $f:\BB=\{q\in\HH:\ |q|<1\}\to\HH$. The function $f$ is \it slice regular \rm \cite{GSAdvances} in the sense of Gentili and Struppa,
who developed a version of complex function theory which holds in the quaternionic setting. See the monograph \cite{libroGSS} for a detailed account of the theory. 
The norm (\ref{htwo}) can be polarized to obtain an inner product with values in the quaternions,
\begin{equation*} 
\left\langle\sum q^na_n,\sum q^nb_n\right\rangle_{H^2(\BB)}:=\sum_{n=0}^\infty\overline{b_n}a_n.
\end{equation*}
The space $H^2(\BB)$ is a reproducing kernel Hilbert space, in the quaternionic sense: for $w$ in $\BB$ and $f$ in $H^2(\BB)$ we have 
\begin{equation*} 
f(w)= \left\langle f,k_w\right\rangle_{H^2(\BB)}, \text{ where } k_w(q)=k(w,q)=\sum_{n=0}^\infty q^n\overline{w}^n.
\end{equation*}
There is a rich interplay between reproducing kernel Hilbert spaces and distance functions. See \cite{distance} for an overview and several examples from 
one-variable holomorphic function space theory. In \cite{cowen} the connection between metric theory and operator theory is analyzed at a very deep level, and the case of 
the Hardy space is a model example of that.  The seminal article \cite{aron} by Aronszajn is still an excellent introduction to the theory of reproducing kernel Hilbert spaces.

In this article we are mainly interested in studying metrics on $\BB$ which are associated with the function space $H^2(\BB)$. We also provide evidence that the 
metric properties of the space 
reflect the  behavior of functions in $H^2(\BB)$. The first metric we consider measures the distance between projections of kernel functions 
in the unit sphere of the Hilbert space $H^2(\BB)$:
\begin{equation}\label{campane}
\delta(p,q):=\sqrt{1-\left|
\left\langle 
\frac{k_q}{\|k_q\|_{H^2(\BB)}},\frac{k_p}{\|k_p\|_{H^2(\BB)}}
\right\rangle_{H^2(\BB)}
\right|^2}.
\end{equation}
In the holomorphic case of $H^2(\Delta)$ one obtains this way the pseudo-hyperbolic metric $\delta^\prime(z,w)=\left|\frac{z-w}{1-\overline{w}z}\right|$.  
A calculation, see Proposition \ref{pseudo} below, gives a formally similar result in the quaternionic case:
\begin{equation*} 
 \delta(z,w)=|(1-q\overline{w})^{-*}*(q-w)|_{|_{q=z}},
\end{equation*}
for $z,w$ in $\BB$. Here, the product $f(q)*g(q)$ and the multiplicative inverse $f(q)^{-*}$ are not pointwise product and pointwise inverse: they are $*$-product and
$*$-inverse, which are defined so that the usual convolution rule for coefficients of power series' products holds. See \cite{libroGSS}, and Section  \ref{arimatea} where we
summarize some background material on slice regular functions.

The infinitesimal version of the pseudo-hyperbolic metric in the complex disc, is the hyperbolic metric in the Riemann-Beltrami-Poincar\'e disc model: 
$ds^2=\frac{|dz|^2}{(1-|z|^2)^2}$. By infinitesimal version of a distance $\delta_0$, we mean the length metric associated with $\delta_0$ (see e.g. \cite{gromov}).
The infinitesimal metric associated with $\delta$ is a Riemannian metric $g$ on $\BB$. In Theorem \ref{metricadiagonale} a formula is produced, which works for a wide class of
reproducing kernel quaternionic Hilbert spaces. Here is the special case of the quaternionic Hardy space.
\begin{teo}\label{riemannian}
{\bf(I) } The length metric associated with (\ref{campane}) is the Riemannian metric $g$ defined below.

For any $w\in \B$, let us identify the tangent space $T_w\B$ with $\HH$. For any vector $d\in T_w\B$, where $w=x+yI_w$ lies in $L_{I_w}$, we decompose
$d=d_1+d_2$ with $d_1$ in $L_{I_w}$ and $d_2$ in $L_{I_w}^\perp$, the orthogonal complement of $L_{I_w}$ with respect to the Euclidean metric in $\HH$. 
The length of $d$ with respect to $g$ is:
\begin{equation}\label{metricacartesianaone}
 |d|_{g(w)}^2=\frac{1}{(1-|w|^2)^2}|d_1|^2+\frac{1}{|1-w^2|^2}|d_2|^2.
\end{equation}

\noindent {\bf(II)} The isometry group of $g$ is the one generated by the following classes of self-maps of $\BB$:
\begin{itemize}
 \item[(a)] regular M\"obius transformations of the form 
 $$
 q\mapsto M_\lambda(q)=(1-q\lambda)^{-*}*(q-\lambda)=\frac{q-\lambda}{1-\lambda q},
 $$
 with $\lambda$ in $(-1,1)$;
 \item[(b)] isometries of the sphere of the imaginary units, 
 $$
 q=x+yI\mapsto T_A(q)=x+yA(I),
 $$
 where $A:\SS\to\SS$ is an isometry of  $\, \SS$;
 \item[(c)] the reflection in the imaginary hyperplane,
 $$
 q\mapsto R(q)=-\overline{q}.
 $$
\end{itemize}
\end{teo}
In the metric \eqref{metricacartesianaone}, the first, ``large''  summand is the hyperbolic metric on a slice, while the second ``small'' summand is peculiarly quaternionic: it measures 
the variation of a quaternionic Hardy function in the direction orthogonal to the slices. Its small size reflects in quantitative, geometric terms the fact that regular functions are
affine in the $\SS$ variable, see \cite{libroGSS}.

The special r\^ole of the real axis in the theory of slice regular functions is here reflected in the fact that all isometries of the metric $g$ fix $\RR\cap\BB$. In particular,
contrary to the case of the complex disc, the action is not transitive. Other, more precise, properties of the metric will be stated and proved on route to the proof of
Theorem \ref{riemannian}. We will study geodesics, geodesically complete submanifolds and other geometric properties of the metric. For instance, we will prove that
the radius of injectivity is infinite at points of the real diameter of $\BB$, and finite elsewhere. The metric has neither positive, nor negative sectional curvature.

The proof of Theorem \ref{riemannian} is split into two steps. In Theorem \ref{metricadiagonale} we will compute the Riemannian metric associated with rather general 
reproducing kernel quaternionic Hilbert spaces $\Hi$; that is the length metric associated with the distance function
\begin{equation*} 
\delta_{\Hi}(w,z)=\sqrt{1-\left|
\left\langle 
\frac{k_w}{\|k_w\|_{\mathcal H}},\frac{k_z}{\|k_z\|_{\mathcal H}}
\right\rangle_{\mathcal H}\right|^2}.
\end{equation*}
We restrict most of our analysis to spaces of functions defined on symmetric slice domanis in $\HH$, with a slice preserving reproducing kernel. Examples are the Hardy space
on $\BB$ and on the right half-space $\HH^+=\{q\in \HH \, | \, \RRe(q)>0 \}$ and the Bergman space on $\BB$. The classification of the isometries is carried out in Section \ref{nonvaria}.

The Riemannian metric $g$ has a rather small group of isometries, compared with the state of 
things in the unit disc of the complex plane, or even in the unit ball in several complex variables, with the 
Bergman-Kobayashi metric. The latter metrics have a transitive group of isometries and, more, the space is isotropic; whereas all isometries of the former
have to fix the real line. One might expect that something better is possible. Unfortunately, there is no Riemannian metric on $\BB$ 
which is invariant under \it regular \rm M\"obius functions and which is ``democratic'' with respect to the sphere of imaginary units. If a geometric invariant for slice 
regular functions on the quaternionic ball exists, it has to be other than a Riemannian metric.
\begin{teo}\label{maancheno}
There is no Riemannian metric $m$ on $\BB$ having as isometries:
\begin{itemize}
 \item[(i)] regular M\"obius transformations $q\mapsto(1-q\overline{a})^{-*}*(q-a)u$, with $a$ in $\BB$ and $|u|=1$;
 \item[(ii)] maps of the form $q=x+yI\mapsto x+yA(I)$, with $A$ in $O(3)$, the orthogonal group of $\RR^3$.
\end{itemize}
\end{teo}
The proof will be given in Subsection \ref{incontri}. 
We mention here that Bisi and Gentili proved in \cite{cinzia} that the usual Poincar\'e metric on $\BB$ is invariant under classical (non-regular) M\"obius maps.

A first relationship between the space $H^2(\BB)$ and the metric $g$ concerns the $H^2$ norm itself.
Let $r\SS^3$ be the sphere of radius $0<r<1$ in $\BB$, with respect to the usual Euclidean metric; containing quaternions $q=re^{tI}$, with $I$ in $\SS$ and $t$ in $[0,\pi]$.
The restriction of $g$ to  $r\SS^3$  induces a volume form $dVol_{r\s^3}$. Let $f$  be in $H^2(\BB)$. Then,
\begin{equation*} 
  \|f\|_{H^2(\BB)}^2=\lim_{r\to1}(1-r^2)\frac{1}{Vol_{r\s^3}(r\s^3)}\int_{r\s^3}|f|^2 dVol_{r\s^3},
\end{equation*}
a relation similar to the definition of the Hardy norm in the unit disc by means of the Poincar\'e metric.

In Section \ref{destra} we use the Caley map $C:q\mapsto(1-q)^{-1}(1+q)$ to write down the metric $g$ in coordinates living in right half-space 
$\HH^+:=C(\BB)=\{q\in \HH \, | \, \RRe(q)>0\}$. This makes it easier to
prove a bilateral estimate for the distance function associated with $g$, Theorem \ref{bilaterale}. As an application, in Theorem \ref{castelguelfo} we
further investigate the ``rigidity'' of the metric $g$, by showing that the only inner functions which are Lipschitz continuous with respect to $g$
have to be be slice preserving.
In particular, they have to fix the real diameter of $\BB$.
 A function defined from $\Omega\subseteq\HH$ to $\HH$ is slice preserving if it maps $L_I\cap\Omega$ to $L_I$ for all $I$ in $\SS$. 

We also consider in Subsection \ref{hardypiatto} four equivalent definitions of the Hardy space $H^2(\HH^+)$ on $\HH^+$. 
First, functions $f$ in $H^2(\HH^+)$
might be characterized, pretty much as in the one-dimensional complex case, as inverse Fourier transforms -in the quaternionic sense- 
of functions $F:[0,\infty)\to\HH$ with finite $L^2$-norm
$$
\|F\|_{L^2}=\left(\int_0^\infty |F(\zeta)|^2d\zeta\right)^{1/2}.
$$
Equivalently, $H^2(\HH^+)$ is the Hilbert space having reproducing kernel $k_w(q)=(q+\overline{w})^{-*}$. This second viewpoint has the advantage of relating
$H^2(\HH^+)$ and $H^2(\BB)$. We show in fact in Proposition \ref{donatini} that the reproducing kernel for $H^2(\HH^+)$ 
is a \it rescaling \rm of the reproducing kernel for $H^2(\BB)$:
$$
k_{H^2(\B)}(C^{-1}(w),C^{-1}(z))=\frac{1}{2}(1+z)k_{H^2(\HH^+)}(w,z)(1+\overline w).
$$
Here we use the symbols $k_{H^2(\BB)}$ and $k_{H^2(\HH^+)}$ for the reproducing kernels on $\BB$ and $\HH^+$, respectively.
Hence, third, the functions in $H^2(\HH^+)$ might be defined as the rescaled versions of functions in $H^2(\BB)$. 
Fourth, the norm of $f$ in $H^2(\HH^+)$  can be computed as the limit of the integrals of $|f|^2$ on ``horocycles'' in $H^2(\HH^+)$, when these are endowed with the 
natural volume form induced by the metric $g$.
The space $H^2(\HH^+)$ was defined in  \cite{semispazio}, using a fifth definition, which is shown to give rise 
to the same reproducing kernel. Our contribution here is mainly relating 
$H^2(\HH^+)$ and the geometry of $\HH^+$.

\section{Preliminaries}\label{arimatea}

We recall the definition of slice regularity, together with some basic results that hold for this class of functions. 
We refer to the book \cite{libroGSS} for all details and proofs.
Let $\HH$ denote the four-dimensional (non-commutative) real algebra of quaternions and 
let $\s$ denote the two-dimensional sphere of imaginary units of $\HH$, $\s=\{q\in \HH \, | \, q^2=-1\}$. 
One can ``slice'' the space $\HH$ in copies of the complex plane that intersect along the real axis,  
\[ \HH=\bigcup_{I\in \s}(\rr+\rr I),  \hskip 1 cm \rr=\bigcap_{I\in \s}(\rr+\rr I),\]
where $L_I:=\rr+\rr I\cong \C$, for any $I\in\s$.  
Then, each element $q\in \HH$ 
can be expressed as $q=x+yI_q$, where $x,y$ are real (if $q\in\rr$, then $y=0$) and $I_q$ is an imaginary unit.  
Let $\Omega\subseteq \HH$ be a subset  of $\HH$. For any $I\in \s$, we will denote by $\Omega_I$ the intersection $\Omega \cap L_I$.
We can now recall the definition of slice regular functions, in the sequel simply called {\em regular} functions.
\begin{defn}
Let $\Omega$ be a domain (open connected subset) in $\HH$. 
A function $f:\Omega \to \HH$ is said to be {\em (slice) regular} if for any $I\in\s$ the restriction $f_I$ of $f$ to $\Omega_I$
has continuous partial derivatives and it is such that
\[\overline{\p}_If_I(x+yI)=\frac{1}{2}\left(\frac{\p}{\p x}+I\frac{\p}{\p y}\right)f_I(x+yI)=0\]
for all $x+yI\in \Omega_I$.
\end{defn}
\noindent A wide class of examples of regular functions is given by power series with quaternionic coefficients of the form $\sum^{\infty}_{n= 0}q^na_n$
which converge in open balls centered at the origin. 
\begin{teo}
A function $f$ is regular on $B(0,R)=\{q\in\HH \, | \, |q|<R\}$ if and only if $f$ has a power series expansion
$f(q)=\sum^{\infty}_{n= 0}q^na_n$ converging in $B(0,R)$.
\end{teo} 
For regular functions, it is possible to define an appropriate notion of derivative:
\begin{defn}
Let $f$ be a regular function on a domain $\Omega\subseteq \HH$. The {\em slice (or Cullen) derivative} of $f$ is the regular function defined as 
\[\de f(x+yI)=\frac{1}{2}\left(\frac{\p}{\p x} -I\frac{\p}{\p y}\right)f_I(x+yI).\]
\end{defn}
We will consider domains in certain restricted classes. 
\begin{defn}
Let $\Omega\subseteq \HH$ be a domain. 
\begin{enumerate}
\item $\Omega$ is called a {\em slice domain} if it intersects the real axis and if, for any $I\in \s$, $\Omega_I$ is a domain in $L_I$.
\item $\Omega$ is called a {\em symmetric domain} if for any point $x+yI\in \Omega$, with $x,y \in \rr$ and $I\in \s$, the entire two-sphere $x+y\s$ is contained in $\Omega$. 
\end{enumerate}
\end{defn}
\noindent The ball $\BB$ and the right half-space $\HH^+=\{q=x+Iy:\ I\in\SS,\ x>0,\ y\in\RR\}$ are symmetric slice domains.

Slice regular functions defined on symmetric slice domains have a peculiar property.
\begin{teo}[Representation Formula]\label{RF}
Let $f$ be a regular function on a symmetric slice domain $\Omega$ and let $x+y\s\subset \Omega$. Then, for any $I,J\in\s$,
\[f(x+yJ)=\frac{1}{2}[f(x+yI)+f(x-yI)]+J \frac I 2 [f(x-yI)-f(x+yI)].\]
In particular, there exist $b,c\in\HH$ such that $f(x+yJ)=b+Jc$ for any $J\in\s$.
\end{teo}
When restricted to a sphere of the form $x+y\s$, a regular function  is actually affine in the variable $q$. This nice geometric property leads to the following definition
\begin{defn}
Let $f$ be a regular function on a symmetric slice domain $\Omega$. The {\em spherical derivative} of $f$ is defined as
\[\p_s f(q)=(q-\overline{q})^{-1}\left(f(q)-f(\overline{q})\right).\]
\end{defn} 
A basic result that establishes a relation between regular functions and holomorphic functions of one complex variable is the following.
\begin{lem}[Splitting Lemma]\label{split}
Let $f$ be a regular function on a slice domain $\Omega\subseteq \HH$. Then for any $I\in\s$ and for any $J\in \s$, $J\perp I$ there exist two holomorphic functions $F,G:\Omega_I\to L_I$ such that
\[f(x+yI)=F(x+yI)+G(x+yI)J\]
for any $x+yI\in\Omega_I$.
\end{lem} 
In general, the pointwise product of functions does not preserve slice regularity. 
It is possible to introduce a new multiplication operation, which, in the special case of power series, can be defined as follows.
\begin{defn}
Let $f(q)=\sum_{n=0}^{\infty}q^na_n,$ and $g(q)=\sum_{n=0}^{\infty}q^nb_n$ be regular functions on $B(0,R)$.
Their {\em regular product} (or {\em $*$-product}) is
\[
f*g(q)=\sum_{n\ge0}q^n \sum_{k=0}^{n}a_kb_{n-k},\]
regular on $B(0,R)$ as well.
\end{defn}
\noindent The $*$-product is related to the standard pointwise product by the following formula.
\begin{pro}\label{trasf}
Let $f,g$ be regular functions on a symmetric slice domain $\Omega$. Then
\[f*g(q)=\left\{\begin{array}{l r}
0 & \text{if $f(q)=0 $}\\
f(q)g(f(q)^{-1}qf(q)) & \text{if $f(q)\neq 0 $}
                \end{array}\right.
 \] 
\end{pro}
\noindent The reciprocal $f^{-*}$ of a regular function $f$ with respect to the $*$-product can be defined. 
\begin{defn}\label{invstar}
Let  $f(q)=\sum_{n=0}^{\infty}q^na_n$ be a regular function on $B(0,R)$, $f\not \equiv 0$.
Its {\em regular reciprocal} is 
\[f^{-*}(q)=\frac{1}{f*f^c(q)}f^c(q),\]
where $f^c(q)=\sum_{n=0}^{\infty}q^n \overline{a}_n$. The function $f^{-*}$ is regular on $B(0,R) \setminus \{q\in B(0,R) \ | \ f*f^c(q)=0\}$ and $f*f^{-*}=1$ 
there.
\end{defn}
\noindent For example, in the case of the {\em reproducing kernel} for the quaternionic Hardy space $H^2(\B)$, we have
\begin{oss}
The {\em reproducing kernel} for $H^2(\B)$ is
\[k_{w}(q)=\sum_{n=0}^{\infty}q^n\overline{w}^n=(1-q\overline{w})^{-*}.\]
\end{oss}
Then we have a natural definition of \emph{regular quotients} of regular functions, which satisfy
\begin{pro}\label{Caterina} 
Let $f$ and $g$ be regular functions on a symmetric slice domain $\Omega$ and denote by $Z=\{q\in \Omega \ | \ f*f^c(q)=0\}$.  
If \  $T_f : \Omega \setminus Z \rightarrow  \Omega \setminus$ is defined as
\[T_{f}(q)=f^c(q)^{-1}qf^c(q),\] then 
\[f^{-*}*g(q)=f(T_{f}(q))^{-1}g(T_{f}(q)) \quad \text{for every} \quad q \in \Omega \setminus Z_{f^s}.\] 
\end{pro}
Important examples of regular quotients that will appear in the sequel are the {\em regular M\"obius transformations}, of the form
\[M_a(q)=(1-q\overline{a})^{-*}*(q-a),\]
where $a\in \B$, which are regular self-maps of the quaternionic unit ball $\B$. After multiplication on the right by unit-norm quaternions, they are the only
self-maps of $\BB$ which are regular, with regular inverse. They were introduced by Stoppato in \cite{stoppato}. See also \cite{libroGSS}.

\section{Metrics associated with quaternionic reproducing kernel Hilbert spaces}
Let $\Omega\subseteq \HH$ be a symmetric slice domain and let $\mathcal H$ be a reproducing kernel Hilbert space of regular functions on $\Omega$. 
For the definition and all basic results concerning quaternionic Hilbert spaces  see, e.g., \cite{ghilonimorettiperotti} and references therein. 
For the properties we are interested in, the same results hold in 
quaternion valued Hilbert spaces and complex valued Hilbert spaces, and the proofs are very similar.
It is possible to define a metric $\delta_{\Hi}$ on $\Omega$ in terms of the distance between projections of kernel functions in the unit sphere of the Hilbert space $\mathcal H$.
Namely, if  $k(w,q)=k_w(q)$ denotes the reproducing kernel of $\mathcal H$, then $\delta_{\Hi}:\Omega\times \Omega\to \rr^+$ can be defined as
\begin{equation}\label{deltaH}
\delta_{\Hi}(w,z)=\sqrt{1-\left|
\left\langle 
\frac{k_w}{\|k_w\|_{\mathcal H}},\frac{k_z}{\|k_z\|_{\mathcal H}}
\right\rangle_{\mathcal H}\right|^2}.
\end{equation}

\begin{pro}\label{metricaHfinita}
Let $\Omega$ be a symmetric slice domain and let $\mathcal H$ be a reproducing kernel Hilbert space of regular functions on $\Omega$. Let $w\in \Omega \cap L_{I_w}$ and let $d \in \HH$ be such that $w+d \in \Omega$. Consider the decomposition $d=d_1+d_2$, where $d_1\in L_{I_w}$ and $d_2\in L_{I_w}^{\perp}$.
Then 
\[\delta^2_{\Hi}(w, w+d)=\frac{\|k_w\|_{\Hi}^2\left\|\overline{d_1\de k_q(w)+d_2\p_s k_q(w)}\right\|_{\Hi}^2-\left|\left\langle k_w, \overline{d_1\de k_q(w)+d_2\p_s k_q(w)}\right\rangle_{\mathcal H}\right|^2}{\|k_w\|_{\mathcal H}^4}+O(|d|^2). \]
\end{pro} 
\begin{proof}

Recalling the definition  \eqref{deltaH} of $\delta_{\Hi}$, we get
\begin{equation}\label{deltaw+d}
\delta^2_{\Hi}(w,w+d)=\frac{\|k_w\|^2_{\mathcal H}\|k_{w+d}\|^2_{\mathcal H}-\left|
\left\langle 
k_w, k_{w+d}\right\rangle_{\mathcal H}\right|^2}{\|k_w\|^2_{\mathcal H}\|k_{w+d}\|^2_{\mathcal H}}.
\end{equation}
We want to have a better description of the numerator of \eqref{deltaw+d}.
Using the properties of the kernel functions and the fact that regular functions are real analytic functions of $4$ real variables, we can write
\[k_{w+d}(q)-k_w(q)=\overline{k_{q}(w+d)-k_q(w)}=\overline{(k_q)_{*}[w](d)}+O(|d|^2)\]
where $(k_q)_{*}[w](d)$ 
denotes the real differential of $k_q$ at the point $w$, applied to the vector $d$. We identify here the tangent space $T_{w}\Omega$ with $\HH$. 
Thanks to the decomposition properties of the real differential of regular functions in terms of slice and spherical derivatives, see Remark 8.15 in \cite{libroGSS}, we have
\[k_{w+d}(q)-k_w(q)=\overline{d_1\de k_q(w)+d_2\p_s k_q(w)}+O(|d|^2),\]
hence,
\[
\|k_{w+d}\|_{\Hi}^2= \|k_w\|_{\Hi}^2+\left\|\overline{d_1\de k_q(w)+d_2\p_s k_q(w)}\right\|_{\Hi}^2+2\RRe \left \langle k_w,\overline{d_1\de k_q(w)+d_2\p_s k_q(w)} \right\rangle_{\Hi}+O(|d|^2)
\]
and 
\begin{equation*}
\begin{aligned}
&\left|\left\langle k_w, k_{w+d}\right\rangle_{\mathcal H}\right|^2= \left|\|k_w\|_{\Hi}^2+\left\langle k_w, \overline{d_1\de k_q(w)+d_2\p_s k_q(w)}\right\rangle_{\mathcal H}\right|^2 +O(|d|^2)\\
&=\|k_w\|_{\Hi}^4+\left|\left\langle k_w, \overline{d_1\de k_q(w)+d_2\p_s k_q(w)}\right\rangle_{\mathcal H}\right|^2\\
& \hskip 6 cm 
+ 2\|k_w\|_{\Hi}^2\RRe\left\langle k_w, \overline{d_1\de k_q(w)+d_2\p_s k_q(w)}\right\rangle_{\mathcal H}+O(|d|^2).
\end{aligned}
\end{equation*}
Therefore
\[\delta^2_{\Hi}(w, w+d)=\frac{\|k_w\|_{\Hi}^2\left\|\overline{d_1\de k_q(w)+d_2\p_s k_q(w)}\right\|_{\Hi}^2-\left|\left\langle k_w, \overline{d_1\de k_q(w)+d_2\p_s k_q(w)}\right\rangle_{\mathcal H}\right|^2}{\|k_w\|_{\mathcal H}^4}+O(|d|^2). \]
\end{proof}
Proposition \ref{metricaHfinita} reflects what happens in the complex case, see \cite{McCarthy}. 
In fact the functions $\overline{\de k_q(w)}$ and $\overline{\p_s k_q(w)}$ are regular with respect to the variable $q$, and they reproduce respectively 
the slice and the spherical derivative of any regular function $f:\Omega \to \HH$. In fact, for any $w\in \Omega_{I_w}$, if  $h \in L_{I_w}$, we can write 
\begin{equation*}
\begin{aligned}
&\de f (w)
=\lim_{h\to 0,\, h\in L_{I_w}} h^{-1} (f(w+h)-f(w))=\lim_{h\to 0,\, h\in L_{I_w}}h^{-1}\left(\left\langle f, k_{w+h} \right\rangle_{\Hi}-\left\langle  f, k_{w} \right\rangle_{\Hi}\right)\\
&=\lim_{h\to 0,\, h\in L_{I_w}} h^{-1}\left\langle  f, \overline{k_q(w+h)-k_q(w)} \right\rangle_{\Hi} =\lim_{h\to 0,\, h\in L_{I_w}} h^{-1}\left\langle  f, \overline{h\de k_q(w)}\right\rangle_{\Hi}
=\left\langle f, \overline{\de k_q(w)} \right\rangle_{\Hi}
\end{aligned}
\end{equation*}
and
\begin{equation*}
\begin{aligned}
\p_s f(w)&=(w-\overline w)^{-1}(f(w)-f(\overline w))=(w-\overline w)^{-1}\left(\langle f,k_w\rangle_{\Hi}-\langle f, k_{\overline w} \rangle_{\Hi}\right)\\
&=(w-\overline w)^{-1} \left\langle f,k_w-k_{\overline w}\right \rangle_{\Hi}
=(w-\overline w)^{-1}\left\langle f, \overline{(w-\overline w)\p_s k_q(w)} \right\rangle_{\Hi}=\left\langle f, \overline{\p_s k_q(w)} \right\rangle_{\Hi}.
\end{aligned}
\end{equation*}

Proposition \ref{metricaHfinita} allows us to define a Riemannian metric $g_{\Hi}$ on the symmetric slice domain $\Omega$. 
For each point $w\in \Omega$, let us identify the tangent space $T_w\Omega$ with $\HH=L_{I_w}+ L_{I_w}^{\perp}$. Then the length of a tangent vector $d=d_1+d_2 \in L_{I_w}+ L_{I_w}^{\perp}$ is 
\begin{equation}\label{metricaH}
|d|^2_{{g_{\Hi}}(w)}=\frac{\|k_w\|_{\Hi}^2\left\|\overline{d_1\de k_q(w)+d_2\p_s k_q(w)}\right\|_{\Hi}^2-\left|\left\langle k_w, \overline{d_1\de k_q(w)+d_2\p_s k_q(w)}\right\rangle_{\mathcal H}\right|^2}{\|k_w\|_{\mathcal H}^4}.
\end{equation}

\begin{teo}\label{metricadiagonale}
Let $\Omega$ be a symmetric slice domain and let $\Hi$ be a reproducing kernel Hilbert space of regular functions on $\Omega$. 
Suppose that $k$ is slice preserving: for any $w\in \Omega$ the kernel function $k_w$ preserves the slice $L_{I_w}$ identified by $w$. 
Then the length of a tangent vector $d=d_1+d_2 \in L_{I_w}+ L_{I_w}^{\perp}\cong T_w\Omega$ with respect to the Riemannian metric $g_{\Hi}$ associated with $\Hi$ is given by
\[|d|^2_{{g_{\Hi}}(w)}=\frac{\left(\|k_w\|_{\Hi}^2\left\|\overline {\de k_q(w)}\right\|_{\Hi}^2-\left|\overline {\de k_w(w)}\right|^2\right)}{\|k_w\|_{\mathcal H}^4}|d_1|^2+\frac{
\left(\|k_w\|_{\Hi}^2\left\|\overline {\p_s k_q(w)}\right\|_{\Hi}^2-\left|\overline {\p_s k_w(w)}\right|^2\right)}{\|k_w\|_{\mathcal H}^4}|d_2|^2.\]
\end{teo}
\begin{proof}
We begin by working out the numerator in equation \eqref{metricaH}. We have:
\begin{equation*}
\begin{aligned}
&\left\|\overline{d_1\de k_q(w)+d_2\p_s k_q(w)}\right\|_{\Hi}^2=\left\langle \overline{d_1\de k_q(w)+d_2\p_s k_q(w)},\overline{d_1\de k_q(w)+d_2\p_s k_q(w)}\right \rangle_{\Hi}\\
&=\left\|\overline {\de k_q(w)}\right\|_{\Hi}^2|d_1|^2+\left\|\overline {\p_s k_q(w)}\right\|_{\Hi}^2|d_2|^2+2\RRe\left \langle \overline {d_1\de k_q(w)}, \overline {d_2\p_s k_q(w)}\right \rangle_{\Hi}
\end{aligned}
\end{equation*}
and
\begin{equation*}
\begin{aligned}
&\left|\left\langle k_w, \overline{d_1\de k_q(w)+d_2\p_s k_q(w)}\right\rangle_{\mathcal H}\right|^2=\left|\overline{d_1\de k_w(w)+d_2\p_s k_w(w)}\right|^2\\
&=\left|\overline {\de k_w(w)}\right|^2|d_1|^2+\left|\overline {\p_s k_w(w)}\right|^2|d_2|^2+2\RRe\left(\overline{d_1\de k_w(w)}d_2\p_sk_w(w)\right).
\end{aligned}
\end{equation*}
Hence we are left to prove that both 
\[\RRe\left \langle \overline {d_1\de k_q(w)}, \overline {d_2\p_s k_q(w)}\right \rangle_{\Hi}=\RRe \left(d_2\left \langle \overline {\de k_q(w)}, \overline {\p_s k_q(w)}\right \rangle_{\Hi}\overline{d_1}\right)\]
and 
\[\RRe\left(\overline{d_1\de k_w(w)}d_2\p_sk_w(w)\right)=\RRe\left( d_2\p_sk_w(w)\overline{\de k_w(w)}\,\overline{d_1}\right)\]
equal zero. 
Now notice that if $k_w$ maps $\Omega_{I_w}$ to $L_{I_w}$, the same holds true for both $\de k_w$ and $\p_s k_w$. The fact that $d_1\in L_{I_w}$ and $d_2\in L_{I_w}^{\perp}$ leads us to conclude.   

\end{proof}
\noindent The hypothesis about kernel functions required in Theorem \ref{metricadiagonale} is  
satisfied by the quaternionic analogues of Hardy and Bergman spaces; see \cite{milanesi2,bergman}.

\section{Invariant metrics associated with the Hardy space $H^2(\B)$}\label{nonvaria}

In this section, we turn our attention to the special example of  the Hardy space $H^2(\B)$. 
We will study the corresponding Riemannian metric $g:=g_{H^2(\B)}$.  
Recalling that for any $w$ the kernel function $k_w(q)=\sum^{\infty}_{n= 0}q^n\overline w^n$ preserves the slice $L_{I_w}$, 
we can directly apply Theorem \ref{metricadiagonale} to find the expression of $g$, thus proving the first part of Theorem \ref{riemannian}. 
\begin{pro} 
For any $w\in \B$, let us identify the tangent space $T_w\B$ with $\HH$. For any vector $d\in T_w\B$, if $w$ lies in $L_{I_w}$ and we decompose
$d=d_1+d_2$ with $d_1$ in $L_{I_w}$ and $d_2$ in $L_{I_w}^\perp$, then the length of $d$ with respect to $g$ is given by
\begin{equation}\label{metricacartesiana}
 |d|_{g(w)}^2=\frac{1}{(1-|w|^2)^2}|d_1|^2+\frac{1}{|1-w^2|^2}|d_2|^2.
\end{equation}
\end{pro}
\begin{proof}
The following equalities can, by their nature, be reduced to simple calculations in the complex plane:
\[\|k_w\|_{H^2(\B)}^2=\frac{1}{1-|w|^2}, \quad \left|\overline{\de k_w(w)}\right|^2=\frac{|w|^2}{(1-|w|^2)^4}, \quad \left|\overline{\p_s k_w(w)}\right|^2=\frac{|w|^2}{(1-|w|^2)^2|1-w^2|^2},\]
\[\left\|\overline{\de k_q(w)}\right\|_{H^2(\B)}^2=\Big \|\sum_{n\ge 0}n^2q^n\overline{w}^{n-1}\Big\|_{H^2(\B)}^2=\frac{1+|w|^2}{(1-|w|^2)^3},\]
\[\left\|\overline{\p_s k_q(w)}\right\|_{H^2(\B)}^2=\frac{1}{|w-\overline w|^2}\left(\frac{2}{1-|w|^2}-\frac{1}{1-w^2}-\frac{1}{1-\overline w^2}\right).\]
A direct application of Theorem \ref{metricadiagonale}, then, yields that, with respect to coordinates $(d_1,d_2)\in (L_{I_w},L_{I_w}^{\perp})$,   
\begin{equation*} 
|d|^2_{g(w)}=\frac{1}{(1-|w|^2)^2}|d_1|^2+\frac{1}{|1-w^2|^2}|d_2|^2.
\end{equation*}
\end{proof}
\noindent 
The volume form $dVol_g$ associated with the metric $g$ at any point $w=x_0+x_1i+x_2j+x_3k\in \B$ is then 
\[dVol_g(w)=\frac{dVol_{Euc}(w)}{(1-|w|^2)^{2}|1-w^2|^{2}},\]
where $dVol_{Euc}(w)=dx_0dx_1dx_2dx_3$ is the usual Euclidean volume element.

\begin{pro}\label{pseudo}
Let $\delta:=\delta_{H^2(\B)}$ be defined as in \eqref{deltaH}. For any $w,z\in \B$, $\delta(z,w)$ coincides both with the value at $z$ of the regular M\"obius transformation $M_w$ associated with $w$ and with the vaule at $w$ of the regular M\"obius transformation $M_z$ associated with $z$, namely   
\[\delta(w,z)=\left|(1-q\overline{z})^{-*}*(q-z)\right|_{|_{q=w}}=\left|(1-q\overline{w})^{-*}*(q-w)\right|_{|_{q=z}}.\]
\end{pro} 
\begin{proof}
Let $w,z$ be two points in $\B$.
By Proposition \ref{Caterina}, 
\[|\langle k_w, k_z \rangle_{H^2(\B)}|=|k_w(z)|=|(1-q\overline{w})^{-*}|_{|_{q=z}}=|1-\hat{z}\overline{w}|^{-1} \]
where
$\hat{z}=(1-zw)^{-1}z(1-zw)$, which implies
\[\left|\left\langle\frac{k_{w}}{\|k_w\|_{H^2(\B)}},\frac{k_{z}}{\|k_z\|_{H^2(\B)}} \right\rangle_{H^2(\B)}\right|^2=|1-\hat{z}\overline{w}|^{-2}\left(1-|w|^2\right)\left(1-|z|^2\right).\]
Thus, since $|\hat{z}|=|z|$, we get 
\begin{equation*}
\begin{aligned}
\delta^2(w,z)&=1-\left|\left\langle\frac{k_{w}}{\|k_w\|_{H^2(\B)}},\frac{k_{z}}{\|k_z\|_{H^2(\B)}} \right\rangle_{H^2(\B)}\right|^2\\
&=\left|1-\hat{z}\overline{w}\right|^{-2}\left(\left|1-\hat{z}\overline{w}\right|^2-\left(1-|w|^2\right)\left(1-|z|^2\right)\right)\\
&=\left|1-\hat{z}\overline{w}\right|^{-2}\left(\left(1-\hat{z}\overline{w}\right)\left(1-w\overline{\hat{z}}\right)-\left(1-w\overline{w}\right)\left(1-\hat{z}\overline{\hat{z}}\right)\right)
=\left|1-\hat{z}\overline{w}\right|^{-2}\left(\hat{z}-w\right)\left(\overline{\hat{z}}-\overline{w}\right)\\
&=\left|1-\hat{z}\overline{w}\right|^{-2}\left|\hat{z}-w\right|^2=\left|\left(1-\hat{z}\overline{w}\right)^{-1}\left(\hat{z}-w\right)\right|^2=\left|\left(1-q\overline{w}\right)^{-*}*(q-w)\right|_{|_{q=z}}^2
\end{aligned}
\end{equation*}
where the last equality follows from Proposition \ref{Caterina}.
\end{proof}

The previous relation between the metric $\delta$ (which is the finite version of the metric $g$)  and regular M\"obius transformations, is not unexpected. In fact, as studied in \cite{bisistop}, the real differential $(M_w)_*$ of the regular M\"obius map $M_{w}$ associated with a point $w\in \B_{I_w}$ acts on $L_{I_w}$ by right multiplication by $(1-|w|^2)^{-1}$ and on $L_{I_w}^{\perp}$ by right multiplication by $(1-\overline w^2)^{-1}$. 
Looking at equation \eqref{metricacartesiana},  we see that the coefficients of the metric $g$ at the point $w$ with respect to coordinates $(L_{I_w},L_{I_w}^{\perp})$ coincide in modulus with the components of $(M_w)_*$. Moreover, the fact that $g(w)$ measures vectors in $L_{I_w}$ by multiplying their Euclidean length by $\frac{1}{1-|w|^2}$ means that the restriction of $g$ to a slice $L_I$ 
is the classical Poincar\'e metric in the unit disc $\B_I$. 
%


Using spherical coordinates, $\B=\{re^{tI} \ | \ r\in[0,1), \ t\in[0,\pi], \ I\in \s \},$
%
if $q=re^{tI}$ and we decompose the lenght element
 $dq=dq_1+dq_2 \in L_{I_w}+L_{I_w}^{\perp}$, then, since $dI$ is orthogonal to $I$ (because $I$ is unitary) we have $|d_1|^2=dr^2+r^2dt^2$ and
$|d_2|^2=r^2\sin^2t|dI|^2$
where $|dI|$ denotes the usual two-dimensional sphere round metric on $\s\cong \s^2$. Therefore we get the expression of the metric tensor $ds^2_g$ associated with $g$ in spherical coordinates:
\begin{equation}\label{metricapolare}
ds_g^2=\frac{dr^2+r^2dt^2}{(1-r^2)^2}+\frac{r^2\sin^2 t|dI|^2}{(1-r^2)^2+4r^2\sin^2t}.
\end{equation}
That is, $g$ is a {\em warped product} of the hyperbolic metric $g_{hyp}$ on the complex unit disc with the 
standard round metric $g_{\s}$ on the two-dimensional sphere \cite{warped}.
 
\subsection{Isometries and geodesics of $(\B,g)$}

%
%
From the expression \eqref{metricacartesiana} of $g$, it is clear that three families of functions act isometrically on $(\B,g)$:
 \begin{itemize}
  \item[(a)] regular M\"obius transformation of the form 
  $$
  q\mapsto M_\lambda(q)=(1-q\lambda )^{-*}*(q-\lambda)=\frac{q-\lambda}{1-q\lambda },
  $$
  with $\lambda$ in $(-1,1)$;
  \item[(b)] isometries of the sphere of imaginary units, which in polar coordinates $r\ge0,t\in[0,\pi],I\in\s$ read as
  $$
  q=re^{tI}\mapsto T_A(q)=re^{tA(I)},
  $$
  where $A:\SS\to\SS$ is an isometry of $\SS$;
  \item[(c)] the reflection in the imaginary hyperplane,
  $$
  q\mapsto R(q)=-\overline{q}.
  $$
 \end{itemize}


\noindent Our goal is to prove the following classification result, thus proving the second part of Theorem \ref{riemannian}.
\begin{teo}\label{teoclassificazione}
The group $\Gamma$ of isometries of $(\B, g)$ is generated by maps of type $(a)$, $(b)$ and $(c)$. 
\end{teo}
The proof requires a few steps. To begin with, we identify three classes of totally geodesic submanifolds of $\B$, each one related to a class of isometries.

The first family is the one related to isometries of type $(a)$.
\begin{lem}\label{capodanno}
For any $I\in \s$, the two-dimensional submanifold of $\B$
\[ \B_I=\B \cap L_I= \{re^{tI}\in \B \, | \, r\in [0,1), t\in[0, 2\pi ]\}\]
is totally geodesic. In particular, for any $I\in \s$, $\B_I$ is an hyperbolic disc.
\end{lem}
\begin{proof}
Fix $I\in \s$ and let $g^I_{hyp}$ be the restriction of the metric $g$ to $\B_I$, which is just the classical hyperbolic metric in the unit disc. 
We will show that each geodesic of $(\B_I, g_{hyp}^I)$ is still a geodesic of $(\B,g)$.
Pick two points $w,z$  in $\B_I$ and let $\gamma$ be the (hyperbolic) geodesic in $\B_I$ joining $w$ with $z$, and $\alpha(\tau)=r(\tau)\left(\cos(t(\tau))+\sin(t(\tau))I(\tau) \right)$
be a parametrized curve which joins $w=\alpha(\tau_0)$ with $z=\alpha(\tau_1)$. 
If $\pi_I \left(\alpha\right)$ denotes the piecewise regular curve obtained by projecting $\alpha$ on $\B_I$,
\[\pi_I \left(\alpha\right)(\tau)=r(\tau)\left(\cos(t(\tau))+\sin(t(\tau))I \right),\]
since $|dI|$ is orthogonal to $\B_I$, we conclude   
\[\length(\alpha) \ge \length(\pi_I\left( \alpha\right))\ge \length(\gamma).\]
\end{proof}

The second family of totally geodesic submanifolds is related to isometries of type $(b)$. For any $I\in \s$, we denote by $\mathcal{C}(I)$ the great circle obtained intersecting $\s$ with the plane $L_I^{\perp}$. 
\begin{lem}\label{C(J)}
For any $J\in \s$, the three-dimensional submanifold of $\B$
\[\B(\mathcal C(J))=\{re^{tI}\in \B \, | \, r\in [0,1), t\in[0,\pi ], I\in \mathcal C(J)\}\]
is totally geodesic.
\end{lem}

\begin{proof}
We will prove the statement by showing that the imaginary units identified by all points lying on a same geodesic always belong to the same great circle of $\s$.
More precisely, let $\gamma(\tau) = r(\tau)\left(\cos(t(\tau))+\sin(t(\tau))I(\tau) \right)$ be a parametrized geodesic of $(\B,g)$ such that 
\begin{equation*}
\left\{
\begin{array}{l}
\gamma(\tau_0)=x_0+y_0I_0\\ 
\gamma '(\tau_0)= v_0 + w_0 J_0  
\end{array}
\right.
\end{equation*}
We want to show that, for any $\tau$, the imaginary unit $I(\tau)$ of $\gamma(\tau)$ belongs to the great circle of $\SS$ identified by $I_0$ and $J_0$, namely that, for any $\tau$, $I(\tau)\in  \mathcal{C}:=\mathcal{C} \left(I_0 \times J_0\right). $
Let $\psi:\s \to \s$ be the reflection of $\s$ with respect to $\mathcal C$. 
Then the curve $\tilde \gamma(\tau)= r(\tau)\left(\cos(t(\tau))+\sin(t(\tau))\psi(I(\tau)) \right)$ is a geodesic of $(\B,g)$ such that
\begin{equation*}
\left\{
\begin{array}{l}
\tilde \gamma(\tau_0)
=\gamma(\tau_0) \\
\tilde \gamma '(\tau_0)
= \gamma ' (\tau_0) 
\end{array}
\right.
\end{equation*}
since $\psi$ fixes $I_0$ and $J_0$.
By the uniqueness of geodesics with assigned initial conditions, we get that $\tilde \gamma (\tau)=\gamma(\tau)$ and hence that $\psi$ fixes $I(\tau)$ for any $\tau$. Therefore we conclude that $I(\tau) \in \mathcal C$ for any $\tau \in I$. 
\end{proof}

The third totally geodesic submanifold is the one related to the last class of isometries, type $(c)$. 
\begin{lem}\label{M}
The three-dimensional submanifold of $\B$
\[\B\left(\pi/ 2\right)=\{re^{tI}\in \B \, | \, r\in [0,1), t=\pi/2 , I\in \s \}= \{rI \, | \, r\in [0,1), I\in \s   \}\]
is totally geodesic.
\end{lem} 
\begin{proof}
The statement can be proven following the line of the proof of Lemma \ref{C(J)}. The ingredients are the fact that the map $R: \B \to \B$, $q\mapsto -\overline{q}$ is an isometry which fixes (punctually) $\B(\pi/2)$, and the uniquness of geodesics with assigned initial conditions.   
\end{proof}

Considering the intersection of totally geodesic submanifolds of type $\B(C(J))$ with $\B(\pi/2)$ allows us to identify another family of totally geodesic submanifolds of $\B$.
\begin{coro}\label{D(C(J))}
Let $\mathcal{C}(J)$ be a great circle in $\s$. Then the two-dimensional submanifold $D\left(\pi/2, \mathcal{C}(J)\right)\subset \B(\pi/2)$, defined as
\[D\left(\pi/2, \mathcal{C}(J)\right)=\{re^{tI}\in \B \, | \, r\in [0,1), t=\pi/2, I \in \mathcal{C}(J) \}=\{rI \in \B \, | \, r\in [0,1), I \in \mathcal{C}(J) \},\] 
is totally geodesic.
\end{coro}
\begin{oss}\label{foliation}
Notice that for the two-dimensional submanifold $D\left({\pi}/2, \mathcal{C}(J)\right)$ the following orthogonality relation holds:
\[D\left({\pi}/2, \mathcal{C}(J)\right)\cap \B_J=\{0\} \text{ and } T_0 D\left({\pi}/2, \mathcal{C}(J)\right)=T_0 \B_J^{\perp}.\]
Moreover, applying M\"obius maps of the form $M_{\lambda}$ to $D\left({\pi}/2, \mathcal{C}(J)\right)$, we can extend the orthogonality relation from the origin to all points in $\BB\cap\RR$. In this way we obtain a family of totally geodesic submanifolds  
\[D\left(t, \mathcal{C}(J)\right)=M_{\lambda(t)}\left(D\left({\pi}/2, \mathcal{C}(J)\right)\right)\]
that, for $t\in [0,\pi]$ and $J\in\s/\{\pm 1\}$, defines a foliation of the manifold $\B$.
\end{oss}

\noindent In order to have some understanding of the (global) behavior of the metric $g$, let
us investigate some metric properties of the discs of the type $D\left(\frac{\pi}{2}, \mathcal{C}(J)\right)$. 
Since the imaginary units taken into account belong to $\mathcal C (J)\cong \s^1$, we can change coordinates, 
setting $I=e^{i\theta}$ and $|dI|=d\theta$, so that the metric $g$, on $D\left(\frac{\pi}{2}, \mathcal{C}(J)\right)$, reduces to
\[ds_{D}^2=\frac{dr^2}{(1-r^2)^2}+\frac{r^2d\theta^2}{(1+r^2)^2}.\] 
%
It is actually convenient to parametrize $D\left(\frac{\pi}{2}, \mathcal{C}(J)\right)\subset \II \cong \rr^3$ as a surface of revolution of the form $(\Phi(\rho), \Psi(\rho) \cos \theta, \Psi(\rho)\sin \theta)$, where $\rho$ is the arc length of the generating curve. 
Setting 
\[\rho=\rho(r)=\frac{1}{2}\log \frac{1+r}{1-r},\]
we get
\[\frac{dr^2}{(1-r^2)^2}=d\rho^2 \quad \text{and} \quad \frac{r^2}{(1+r^2)^2}=\frac{1}{4} \tanh^2 (2 \rho)\]
and hence, in coordinates $(\rho, \theta)$, we get that the metric is expressed as
\begin{equation}\label{malanotteno}
ds_{D}^2=d\rho^2+\frac{1}{4} \tanh^2 (2 \rho)d\theta^2=d\rho^2+\Psi^2(\rho) d\theta^2.
\end{equation}
\begin{oss}
The Gaussian curvature $K$ of the two-dimensional submanifold $D\left(\frac{\pi}{2}, \mathcal{C}(J)\right)$ is positive. 
In fact, 
see e.g.  \cite{DoCarmo}, with respect to coordinates $(\rho, \theta)$ it can be computed as
\[K=\frac{-\Psi''(\rho)}{\Psi(\rho)}\]
which is a non-negative quantity since $\Psi(\rho)=\frac{1}{2} \tanh (2 \rho)\ge 0$ and $\Psi''(\rho)\le 0$. 
This in particular implies that the sectional curvature of $(\B,g)$ is positive on all sections $D\left(\frac{\pi}{2}, \mathcal{C}(J)\right)$, while it is negative on all slices $\B_I$.
\end{oss}

It is possible to study geodesics of  $D\left(\frac{\pi}{2}, \mathcal{C}(J)\right)$ by means of the Euler-Lagrange equations 
\begin{equation*}
\left\{\begin{array}{l}
\frac{\partial}{\partial \theta}L=\frac{d}{dt}\frac{\partial}{\partial \dot \theta}L\\ 
\frac{\partial}{\partial \rho}L=\frac{d}{dt}\frac{\partial}{\partial \dot \rho}L
\end{array}
\right.
\end{equation*}
associated with the Lagrangian
\[L(\rho, \theta,\dot \rho, \dot \theta, \tau )=\frac{1}{2}\left(\dot \rho^2+\frac{\tanh^2(2\rho)}{4}\dot \theta^2\right),\]
namely
\begin{equation} \label{geo1}
\left\{\begin{array}{l}
 0=\frac{d}{dt}\left(\frac{\tanh^2(2\rho)}{4} \dot \theta\right)  \\
\tanh (2\rho)\frac{1-\tanh^2(2\rho)}{4}\dot \theta^2=\ddot {\rho}. 
\end{array}
\right.
\end{equation}
The first equation in \eqref{geo1} yields
\[\frac{\tanh^2(2\rho)}{4} \dot \theta= A,\]
for some constant $A$. If $A=0$, we get $\dot \theta=0$ and hence the second equation in \eqref{geo1} implies that $\ddot{\rho}=0$.
If otherwise $A \neq 0$ we get $\dot \theta=\frac{4A}{\tanh^2(2\rho)}$ which implies $|\dot \theta|>4|A|$. 
Therefore all generating curves, with $\dot \theta=0$, are geodesics of $D$. Which is not surprising since they correspond to radii $\gamma(r)=re^{\frac{\pi}{2}I}$ for $I\in \mathcal C(J)$. 
The other important fact that arises is that for any point $q\in D\left(\frac{\pi}{2}, \mathcal{C}(J)\right)\setminus \{0\}$ any geodesic corresponding to $A\neq 0$ intersects in finite time the ``radial'' geodesic through $q$. This leads to the following result.
\begin{lem}\label{inj}
Let $J\in \s$. For any $q\in D\left(\frac{\pi}{2}, \mathcal{C}(J)\right)$ such that $q\neq 0$, the injectivity radius at $q$ is finite. On the other hand, the injectivity radius at $q=0$ is infinite. 
\end{lem}  

\noindent This important metric property of the point $q=0$ is useful to classify the isometries of $(\B,g)$. 
First of all it tells us that isometries map the real diameter of $\BB$ to itself.
\begin{lem}\label{natale}
Let $\Gamma$ be the group of isometries of $(\B, g)$. Then, for any $\phi \in \Gamma$, $\phi(\B\cap \RR)= \B\cap \RR$.
\end{lem}
\begin{proof}
Consider first $q=0$. Since the injectivity radius at $q=0$ is infinite, then, for any $\phi\in \Gamma$, the injectivity radius at $\phi(0)$ is infinite as well. By post-composing $\phi$ with a regular M\"obius transformation of type $(a)$ $M_{\lambda}$ we can map $0$ to $D\left(\frac{\pi}{2}, \mathcal{C}(J)\right)$ (for some $J\in \s$) and hence Lemma \ref{inj} yields that $M_{\lambda}( \phi(0))=0$. 
Since $M_{\lambda}$ preserves the real diameter of $\BB$, we get that $\phi(0)\in \rr$. To conclude, notice that we can map each point of $\B\cap \rr$ to $0$ by means of a regular M\"obius map of type $(a)$.     
\end{proof}

We can finally prove the Classification Theorem for isometries of $(\B, g)$.
\begin{proof}[Proof of Theorem \ref{teoclassificazione}]
Let $\Phi \in \Gamma$ be an isometry of $(\B, g)$. Up to composition with a regular M\"obius transformation of type $(a)$ and with the map $R:q\mapsto -\overline q$, 
we can suppose that $\Phi(0)=0$ and that, by Lemma \ref{natale}, $\Phi(\B\cap \rr^+)=\B\cap \rr^+$. 

\it We now show that $\Phi$ fixes $\B(\pi/2)$. \rm Set $\tilde{B}(\pi/2)=\Phi(\B(\pi/2))$. Since $\Phi$ is an isometry, Lemma \ref{M} implies that $\tilde{B}(\pi/2)$ is a 
totally geodesic submanifold of $\B$. Moreover, since $\Phi(0)=0$, since the geodesics starting at $0$ lie on slices, and since, by Lemma \ref{capodanno}, 
the slices carry the usual hyperbolic-Poincar\'e metric: 
we have that $\Phi$ maps radii $\gamma_I(r)=re^{\frac{\pi}{2}I}$ to radii of the 
form $\Phi(\gamma_I(r))=re^{\theta(I)\psi(I)}$ with $\theta(I)\in [0,\pi]$, and $\psi(I)\in\s$. 
Let us show that $\theta$ is actually constant on $\s$. \\
If $d_g$ denotes the distance function on $\B$ associated with $g$, recalling equation \eqref{metricapolare}, on the one hand we have 
\begin{eqnarray}\label{iso1}
d_{g}\left( \Phi(\gamma_{I_1}(r)),\Phi(\gamma_{I_2}(r))\right)&=&d_{g}\left(\gamma_{I_1}(r),\gamma_{I_2}(r)\right)=
d_{g}\left( re^{\frac{\pi}{2}I_1},re^{\frac{\pi}{2}I_2}\right)\crcr
&\le&\frac{r}{1+r^2}d_{\s}(I_1,I_2)\stackrel{r\to1}{\longrightarrow}\frac{1}{2}d_{\s}(I_1,I_2),
\end{eqnarray}
where $d_{\s}$ denotes the usual spherical distance on the unit sphere $\s$.
In particular we deduce that 
$$
d_{g}\left( \Phi(\gamma_{I_1}(r)),\Phi(\gamma_{I_2}(r))\right)
$$ 
is bounded as a function of $r$.
On the other hand, if $\alpha(\tau)=r(\tau)e^{t(\tau)I(\tau)}$ is a parametrized geodesic joining $\alpha(\tau_1)=\Phi(\gamma_{I_1}(r))$ and 
$\alpha(\tau_2)=\Phi(\gamma_{I_2}(r))$, we have
\begin{equation*} 
\begin{aligned}
&d_{g}\left( \Phi(\gamma_{I_1}(r)),\Phi(\gamma_{I_2}(r))\right)=d_{g}\left( re^{\theta({I_1})\psi({I_1})},re^{\theta({I_2})\psi({I_2})}\right)= 
\length \left(\alpha ([\tau_1,\tau_2])\right)\\
&= \int_{\tau_1}^{\tau_2}\sqrt{\frac{r^\prime(\tau)^2+r^2t^\prime(\tau)^2}{1-r(\tau)^2}+\frac{r(\tau)^2\sin^2(t(\tau))I^\prime(\tau)^2}{(1-r(\tau)^2)^2+4r(\tau)^2\sin^2(t(\tau))}}
d\tau\crcr
&\ge \int_{\tau_1}^{\tau_2}\sqrt{\frac{r^\prime(\tau)^2+r(\tau)^2t^\prime(\tau)^2}{1-r(\tau)^2}}d\tau\ge d_{hyp}(re^{\theta({I_1})I},re^{\theta({I_2})I})
\end{aligned}
\end{equation*}
where $I$ is any fixed imaginary unit and $d_{hyp}$ the hyperbolic distance associated to the restriction $g^I_{hyp}$ of the metric $g$ to $\BB\cap L_I$.
If, by contradiction, $\theta(I_1)\neq \theta(I_2)$, then the distance 
$d_{hyp}(re^{\theta({I_1})I},re^{\theta({I_2})I})$ tends to infinity as $r$ goes to $1$, contradicting (\ref{iso1}).
Then, $\theta(I_1)=t(\tau_1)=t(\tau_2)=\theta(I_2)$: $\theta$ is constant on $\s$.

Therefore we have that $\tilde{B}(\pi/2)$ is ruled by radii of the form $\tilde {\gamma}(r)=re^{t_0 I}$ for some constant $t_0$. 
If $t_0=\pi/2$, then we are done. Suppose then $t_0\ne\pi/2$.
Since $\tilde{B}(\pi/2)$ and $\B(\pi/2)$ intersect at $0$ and they are three-dimensional submanifolds 
in $\HH$, the intersection $V$ of their respective tangent spaces at $0$ must have dimension $2$ or $3$. Let $v$ be a vector in $V$ and let $r\mapsto re^{Jt}$
be  the reparametrized geodesic with initial velocity $v$. The geodesic lies on both $\tilde{B}(\pi/2)$ and $\B(\pi/2)$, hence $t_0=t=\pi/2$. 
(A different proof consists in showing that, if $t_0\ne\pi/2$, then $\tilde{B}(\pi/2)$ is not smooth at the origin).


\it The next step is to show that the restriction of $\Phi$ to $\B(\pi/2)$ is an isometry $T_A$ of type $(b)$ for some isometry $A$ of the sphere $\s$. \rm 
We have that
\begin{equation} \label{iso2}
\begin{aligned}
&d_{g}\left( \Phi(\gamma_{I_1}(r)),\Phi(\gamma_{I_2}(r))\right)= 
d_{g}\left( re^{\frac{\pi}{2}\psi({I_1})},re^{\frac{\pi}{2}\psi({I_2})}\right).
\end{aligned}
\end{equation}
We now prove an improvement of (\ref{iso1}). 
\begin{lem}\label{amaranto}
 $$
 \lim_{r\to1}d_{g}\left( re^{\frac{\pi}{2}I_1},re^{\frac{\pi}{2}I_2}\right)=\frac{1}{2}d_{\s}(I_1,I_2).
 $$
\end{lem}
\begin{proof}[Proof of the lemma.] Only the case $I_1\ne I_2$ is interesting. Let $D\left(\pi/2, \mathcal{C}(J)\right)$ be the two-dimensional
manifold introduced in 
Lemma \ref{D(C(J))} which contains the reparametrized geodesics $r\mapsto re^{\frac{\pi}{2}I_j}$, $j=1,2$. The metric $g$ restricted to the totally geodesic surface 
$D\left(\pi/2, \mathcal{C}(J)\right)$ was discussed earlier in this subsection, where we gave it the expression (\ref{malanotteno}). 
Since $\Psi^\prime(\rho)=1/\cosh(2\rho)\le1$, the surface can be isometrically imbedded as a surface $S$ in $\RR^3$, with parametric equations $(u(s,\theta),v(s,\theta),z(s,\theta))
=\chi(s,\theta)$,
where:
$$
\begin{cases}
 u=p(s)\cos(\theta);\crcr
 v=p(s)\sin(\theta);\crcr
 z=s.
\end{cases}
$$
Here $s\ge0$, $\theta\in[-\pi,\pi]$ and $p:[0,+\infty)\to[0,1/2)$ is a smooth, increasing function such that $p(0)=0$ and $\lim_{s\to\infty}p(s)=1/2$. The relationship between 
$p$ and $\Psi$ is the following: if $\int_0^s\sqrt{p^\prime(\sigma)^2+1}d\sigma=\rho,$ then $p(s)=\psi(\rho)$. Now, $r=constant\to1$ corresponds to $s=constant\to\infty$,
and the choice of $I_1$ and $I_2$ corresponds to a choice of $\theta_1$ and $\theta_2$. Let $k$ be the metric on the surface. It is elementary that 
$$
\lim_{s\to\infty}d_k(\chi(s,\theta_1),\chi(s,\theta_2))=\frac{1}{2}d_{\mathbb{S}^1}(\theta_1,\theta_2)
$$
is one-half the distance between $\theta_1$ and $\theta_2$ on the unit circle, which is the same as one-half the distance between $I_1$ and $I_2$ in $\SS$.
\end{proof}
Equations \eqref{iso1} and \eqref{iso2} together with Lemma \ref{amaranto} imply that that $\psi:\s\to\s$ 
is an isometry of the sphere $\s$, i.e. $\Phi|_{\B(\pi/2)}=T_{\psi}|_{\B(\pi/2)}$.
In conclusion, $T_{\psi}^{-1}\circ \Phi$ is an isometry that fixes $\rr\cap \B$ and $\B(\pi/2)$ and hence its (real) differential at the origin $(T_{\psi}^{-1}\circ \Phi)_*[0]:T_0\B\to T_0\B$ is the identity map. Therefore $T_{\psi}^{-1}\circ \Phi$ is the identity map as well and the theorem is proved. 
\end{proof}
\subsection{Relation with the space $H^2(\B)$}
If we restrict the metric $g$ to a three-dimensional sphere $r\s^3$ of radius $r$, in spherical coordinates we get
\[ds^2_{r\s^3}=\frac{r^2}{(1-r^2)^2}dt^2+\frac{r^2\sin^2 t}{(1-r^2)^2+4r^2\sin^2(t)}|dI|^2
\] 
whose corresponding volume form is
\[dVol_{r\s^3}(re^{tI})=\frac{r^3\sin^2t}{(1-r^2)((1-r^2)^2+4r^2\sin^2(t))}dtdA_{\s}(I)\]
where $dA_{\s}$ denotes the area element of the two-dimensional sphere $\s$.
This volume form (after a normalization) induces a volume form on the boundary $\s^3$ of the unit ball: if $u=e^{sJ}\in\s^3$, we have
\begin{equation*}
\begin{aligned}
dVol_{\s^3}(u)&:=\lim_{r\to 1^-}(1-r^2)dVol_{r\s^3}(ru)=\lim_{r\to 1^-}\frac{(1-r^2)r^3\sin^2s}{(1-r^2)((1-r^2)^2+4r^2\sin^2(s))}dtdA_{\s}(I)\\
&=\frac{1}{4}dtdA_{\s}(I).
\end{aligned}
\end{equation*}
Notice that $dVol_{\s^3}$ is the product of the usual spherical metric on the two-dimensional sphere $\s$ with the metric $dt$ on circles $\s^3_I$ which appears in the definition of Hardy spaces given in \cite{hardy}.
Moreover in \cite{hardy} it is proven that any $f \in H^2(\B)$ has radial limit along almost any radius and hence, denoting (with a slight abuse of notation) the radial limit by $f$ itself, we have  
\begin{equation*}
\begin{aligned}
\int_{\s^3}|f(u)|^2dVol_{\s^3}(u)&=\frac{1}{4}\int_{\s}\int_{0}^{\pi}|f(e^{tI})|^2dtdA_{\s}(I)=\frac{1}{8}\int_{\s}\int_{0}^{2\pi}|f(e^{tI})|^2dtdA_{\s}(I)\\
&=\frac{1}{8}\int_{\s}||f||^2_{H^2(\B)}dA_{\s}(I)=\frac{\pi}{2}||f||^2_{H^2(\B)}.
\end{aligned}
\end{equation*}

\subsection{Proof of Theorem \ref{maancheno}}\label{incontri}
\it We begin by showing that $(\BB,m)$ has constant negative curvature. \rm Heuristically, a Riemannian metric $m$ satisfying the assumption of the theorem has an isometry group with dimension
$$
dim(\BB)+dim(\SS^3)+dim(O(3))=4+3+3=10,
$$
which is maximal for a four-dimensional Riemannian manifold. Hence, $(\BB,m)$ has constant curvature. 

More precisely, we show that the isometry group $\III$ acts transitively 
on orthonormal frames, a property which is known to imply constant sectional curvature. Given points $a,b$ in $\BB$ and orthonormal frames $\{e_l(a):\ l=0,1,2,3\}$ and $\{e_l(b):\ l=0,1,2,3\}$ in 
$T_a\BB$ and $T_b\BB$, respectively,
we find an isometry $\varphi$ in $\III$ such that its (real) differential $\varphi_*$ satisfies: $\varphi_{*}[a]e_l(a)=e_l(b)$. 
We in fact exhibit $\varphi$ mapping $a$ to $0$ and such that $\varphi_*$ maps the chosen orthonormal frame in $T_a\B$ to  a fixed orthonormal basis $e_0(0),e_1(0),e_2(0),e_3(0)$ of $T_0\BB$, where $e_0(0)$ is the vector tangent to the positive real half-axis. 
The isometry
$M_a(q)=(1-q\overline{a})^{-*}*(q-a)$ maps $a$ to $0$, hence $(M_a)_*$ sends $e_l(a)$ to a orthonormal frame in $T_0\B$ $e^\prime_l$ for $l=0,\dots,3$. 
For a suitable choice of $u$ with $|u|=1$, the isometry $q\mapsto q\cdot u$ has differential mapping $e^\prime_0$ to $e_0(0)$, 
and $e^\prime_l(0)$ to $e^{\prime\prime}_l(0)$ ($j=1,2,3$). The isometries $q=x+yI\mapsto T_A(q)=x+yA(I)$, $A$
being a fixed element of $O(3)$, all have differentials fixing $e_0(0)$. We can find one mapping $e^{\prime\prime}_l(0)$ to $e_l(0)$ for $l=1,2,3$. The composition 
of these three isometries is the desired isometry $\varphi$.

For $I\in\SS$, consider the subgroup $\III_I$ of $\III$ of the isometries fixing the slice $\BB_I$; which consists of the regular M\"obius maps $M_a$, with $a$ in $\BB_I$, 
and of the maps $q\mapsto q\cdot e^{tI}$. If $\chi_I: x+yI\mapsto x+yi$ is the natural bijection from $\BB_I$ to the unit disc in the complex plane, $\chi_I\III_I\chi_I^{-1}$
identifies $\III_I$ with the usual Poincar\'e group in the complex disc. Hence, the restriction of $m$ to $\BB_I$ is (isometric to) a constant multiple of the 
Poincar\'e metric, which has constant negative curvature.

\it The hyperbolic metric $m$ is realized by the standard Poincar\'e model on the ball $\BB$. \rm
The metric $m$ restricted to $\BB_I$ is realized as $|d|^2_{m(w)}=\lambda^2\frac{|d|^2}{(1-|w|^2)^2}$ (for $w$ in $\BB_I$ and $d$ in $T_w(\BB_I)$), with $\lambda$
which is independent of $I$, since different slices intersect along the real diameter of $\BB$. We might set $\lambda=1$. Each slice $\BB_I$ is totally geodesic, since 
it is the set of the points fixed by an isometry of the type $x+yJ\mapsto x+yB(J)$, where $B$ is an element of $O(3)$ fixing $\pm I$ and no other element of $\SS$.

Then, the radii $r\mapsto ru=\gamma_u(r)$ (with $u$ fixed in $\HH$, $|u|=1$) are (reparametrizations of) geodesics of $(\B,m)$, not just of its restriction to a slice, and the distance function on each of them
is obtained by integrating $dr/(1-r^2)$. The three-dimensional spheres $r\s^3=\{q:\ |q|=r\}$ are then metric spheres centered at $0$ for the metric $m$. 
By Gauss Lemma, they are orthogonal to the curves $\gamma_u$. Fix $r$ in $(0,1)$. An argument similar to the one above shows that 
the subgroup $\III_0$ of the isometries fixing $0$ acts transitively on the bundle of the orthogonal frames at points of $r\s^3$, hence $r\s^3$ is isometric to the usual
three-spheres with a multiple of the spherical metric. Since $L_I\cap r\s^3$ is isometric to a metric one sphere in the Poincar\'e model of the hyperbolic metric 
(in dimension two), for each $I$
in $\SS$, $r\s^3$ is similarly isometric to a metric three-sphere in the Poincar\'e model of the hyperbolic metric (in dimension four). But we said that $r\s^3$ and $\gamma_u$
are orthogonal in their point of intersection; they are both isometric to the corresponding objects in the Poincar\'e model; the sum of their tangent spaces is the whole tangent space:
this shows that the metric $m$ coincides in fact with the Poincar\'e metric.

\it Concluding, \rm we have shown that the hyperbolic Poincar\'e metric is invariant under regular M\"obius maps, but this contradicts a result of Bisi and Stoppato \cite{bisistop}, 
Remark 5.

\section{Metric in the right half space $\HH^+$}\label{destra}
Consider the right half space $\HH^+=\{q\in \HH \, | \, \RRe(q)>0\}$. The Cayley map $C:\B\to \HH^+$,
\[C(q)=(1-q)^{-1}(1+q),\]
is a regular bijection from the quaternionic unit ball onto the quaternionic right half space with regular inverse the function $C^{-1}:\HH^+\to \B$,
\[C^{-1}(q)=(1+q)^{-1}(q-1).\]
The aim of this section is to study the image $(\HH^+, h)$ of $(\B,g)$ under the map $C$, where $h$ is the pullback of the metric $g$ by the map $C^{-1}$. 
 In the introduction we labeled $g$ and $h$ by the same letter, since $C$ is by definition an isometry from $(\BB,g)$ to $(\HH^+,h)$. 
Let $u\in \HH^+$ and let $v=v_1+v_2$ be a tangent vector in $T_u\HH^+\cong L_{I_u}+L_{I_u}^{\perp}$. The length of $v$ with respect to $h$ is 
\[|v|_{h(u)}=\left|(C^{-1})_*[u](v)\right|_{g(C^{-1}(u))}\]
where $(C^{-1})_*[u]$ is the real differential of $C^{-1}$ at the point $u \in \HH^+$. Recalling the decomposition of the real differential of a regular function in terms of its slice and spherical derivatives, if $v=v_1+v_2\in L_{I_u}+L_{I_u}^{\perp}$ we can write
\[ (C^{-1})_*[u](v)=v_1\de C^{-1}(u)+v_2\p_s C^{-1}(u)=v_1\frac{2}{(1+u)^{2}}+v_2\frac{2}{|1+u|^2}.\]
Hence $(C^{-1})_*[u]$ preserves the decomposition $T_u\HH^+=L_{I_u}+L_{I_u}^{\perp}$ and we get
\begin{equation*}
\begin{aligned}
|v|^2_{h(u)}&=\frac{1}{(1-|C^{-1}(u)|^2)^2}\frac{4}{|1+u|^{4}}|v_1|^2+\frac{1}{|1-C^{-1}(u)^2|^2}\frac{4}{|1+u|^4}|v_2|^2\\
&=\frac{1}{4\RRe(u)^2}|v_1|^2+\frac{1}{4|u|^2}|v_2|^2.
\end{aligned} 
\end{equation*}
If $v\in L_{I_u}$ then its length is, not surprisingly, the hyperbolic length in the hyperbolic half plane $\HH^+_{I_u}=\{x+yI_u \, | \, x>0, y\in \rr\}$. 
Notice that $C$ maps $\B_I$ to $\HH^+_I$ for any $I\in \s$ and it maps the totally geodesic submanifold $\B(\pi/2)$ to $\HH^+(\pi/2):=C\left(\B(\pi/2)\right)=\{q\in \HH^+ \, | \, |q|=1\}$ i.e. the right half of the three-dimensional unit  sphere $\s^3$. Then it is not difficult to verify that the isometry group of $(\HH^+,h)$ is generated by the images under $C$ of isometry of $(\B,g)$ of type $(a)$, $(b)$ and $(c)$, 
 \begin{itemize}
  \item[(a')] linear maps preserving the positive real half-axis, 
  $$
  q\mapsto q\lambda,
  $$
  with $\lambda>0$;
  \item[(b')] isometries of the sphere of the imaginary units, which in polar coordinates $r\ge0$, $t\in[0,\pi/2)$, $I\in\s$ read as
  $$
  q=re^{tI}\mapsto T_A(q)=re^{tA(I)},
  $$
  where $A:\SS\to\SS$ is an isometry of $\SS$;
  \item[(c')] the inversion in the three-dimensional unit (half) sphere 
  $$
  q\mapsto \frac{1}{\overline{q}}.
  $$
 \end{itemize}
Acting on $\HH^+(\pi/2)$ by means of isometries of type $(a')$ we obtain totally geodesic regions of the form $\{q\in \HH^+ \, | \, |q|=R\}$ for $R>0$, that can be sliced in totally geodesic two-dimensional submanifolds, corresponding to submanifolds of type $D(t, C(J))$ in the ball case.

In this setting it is possible to introduce {\em horocycles}, i.e. hyperplanes of points with constant real part, $H_c=\{q\in \HH^+ \, | \, \RRe(q)=c\}$ for some constant $c>0$. They deserve the name of horocycles because their intersection with each slice $L_I$ is a proper horocycle in the hyperbolic half plane $\HH^+_I$. Isometries of type $(a')$ map horocycles one into another.

If we restrict the metric $h$ to horocycles we obtain that the length of a vector $v=v_1+v_2 \in T_uH_c\cong \rr I_u + L_{I_u}^{\perp}$ tangent to the horocycle $H_c$ at the point $u\in H_c$, can be written as
\[|v|^2_{H_c}=\frac{1}{4c^2}|v_1|^2+\frac{1}{4(c^2+|\IIm(u)|^2)}|v_2|^2\]
and the corresponding volume form at $u=c+x_1i+x_2j+x_3k$ is
\[dVol_{H_c}(u)=\frac{dVol_{Euc}(u)}{8c(c^2+|\IIm(u)|^2)},
\] 
(since the component in $L_{I_u}$ is  one-dimensional) where $dVol_{Euc}(u)=dx_1dx_2dx_3$ is the standard Euclidean volume element. 
If we want to define a (non-degenerate) volume form $dVol_{\p\HH^+}$ on the boundary $\p \HH^+$ of the quaternionic right half space we can not directly take the limit of $dVol_{H_c}$ as $c$ approaches $0$, we need indeed first to normalize it. For any $u\in \p \HH^+$, we define
\begin{equation}\label{hpiu}
dVol_{\p\HH^+}(u)\!:=\!\!\lim_{c\to 0^+}c \left(dVol_{H_c}(u+c)\right)=\!\lim_{c\to 0^+}\frac{dVol_{Euc}(u+c)}{8(c^2+|\IIm(u+c)|^2)}=\frac{dVol_{Euc}(u)}{8|\IIm(u)|^2}=\frac{dVol_{Euc}(u)}{8|u|^2}.
\end{equation}

\subsection{Hardy space on $\HH^+$}\label{hardypiatto}
We will show that, as in the case of the metric $g$ on $\B$,  the invariant metric $h$ on $\HH^+$ introduced in the previous section 
and in particular the corresponding volume form \eqref{hpiu}, is related with the quaternionic Hardy space on the right half space $\HH^+$. 
It is possible to define the Hardy space $H^2(\HH^+)$ on $\HH^+$ as the space of regular functions $f:\HH^+\to \HH$ of the form, 
\[f(q)=\int_{0}^{+\infty}e^{-\zeta q}F(\zeta)d\zeta,\]
with $F:\rr^+\to \HH$, such that
\[||f||^2_{H^2(\HH^+)}:=\int_{0}^{+\infty}|F(\zeta)|^2d\zeta<+\infty.\]
With this definition, the reproducing kernel of $H^2(\HH^+)$ is a function
\[k(q,w)=k_w(q)=\int_{0}^{+\infty}e^{-\zeta q}G(\zeta)d\zeta\]
where $G:\rr^+\to \HH$ is such that
\[f(w)=\langle f, k_w\rangle_{H^2(\HH^+)}=\int_{0}^{+\infty}\overline{G(\zeta)}F(\zeta)d\zeta=\int_{0}^{+\infty}e^{-\zeta w}F(\zeta)d\zeta.\]
Hence $G$ has to satisfy $\overline{G(\zeta)}=e^{-\zeta w}$ which implies ${G(\zeta)}=e^{-\zeta \bar w}$, i.e. that the kernel function is
\[k_w(q)=\int_{0}^{+\infty}e^{-\zeta q}e^{-\zeta \bar w}d\zeta.\]
To obtain a closed expression of $k_w(q)$, let first $q$ be a (positive) real number. In this case $q$ commutes with all points in $\HH^+$ and we can write
\[k_w(q)=\int_{0}^{+\infty}e^{-\zeta(q+\bar{w})}d\zeta=
\frac{1}{q+\overline w}.
\]
Consider now the function $q\mapsto (q+\overline w)^{-*}$ (here the regular reciprocal is defined with a slight generalization of Definition \ref{invstar}, see \cite{libroGSS}). 
This function is regular and it coincides with $q\mapsto (q+\overline w)^{-1}$ on real numbers. 
Thanks to the Identity Principle for regular functions, Theorem 1.12 in \cite{libroGSS}, we obtain that the reproducing kernel is
$k_w(q)= (w+\overline q)^{-*}=\int_{0}^{+\infty}e^{-\zeta w}e^{-\zeta \bar{q}}d\zeta$.

Another way to obtain the reproducing kernel on $H^2(\HH^+)$ is the following. 
\begin{pro}\label{donatini}
Denote by $k_{H^2(\B)}$ and by $k_{H^2(\HH^+)}$ the reproducing kernels of the Hardy space on the unit ball $H^2(\B)$ and on the right half-space $H^2(\HH^+)$ respectively. Let $C:\BB\to\HH^+$ 
be the Cayley map, $C(q)=(1-q)^{-1}(1+q)$.
For any $z,w\in \HH^+$, the function 
$k_{H^2(\B)}(C^{-1}(w),C^{-1}(z))$ is a rescaling of the reproducing kernel of $H^2(\HH^+)$:
\[ 
k_{H^2(\B)}(C^{-1}(w),C^{-1}(z))=\frac{1}{2}(1+z)k_{H^2(\HH^+)}(w,z)(1+\overline w).\]
\end{pro}

\begin{proof}
The map $C^{-1}$, having real coefficients, is slice preserving. Hence, we can compose $k_{H^2(\B)}$ with $C^{-1}$ 
preserving (left) regularity in the first variable and (right) ``anti-regularity'' in the second one. 
We have, then,
\begin{equation*}
\begin{aligned}
& k_{H^2(\B)}(C^{-1}(z),C^{-1}(w))=(1-q \overline{C^{-1}(w)})_{|_{q=C^{-1}(z)}}^{-*}\\
&=\left(1-2C^{-1}(z)\RRe(C^{-1}(w))+C^{-1}(z)^2|C^{-1}(w)|^2\right)^{-1}\left(1-C^{-1}(z)C^{-1}(w)\right)\\
&=(1+z)^2|1+w|^2\left(|1+w|^2(1+z)^2-2(1-z^2)(1-|w|^2)+(1-z)^2|1-w|^2\right)^{-1}\\
&\hskip 4 cm \cdot(1+z)^{-1}\left((1+z)(1+w)-(1-z)(1-w)\right)(1+w)^{-1}\\
&=(1+z)\frac{1}{4}\left(z^2+2z\RRe(w)+|w|^2\right)^{-1}2\left(z+w\right)(1+\overline w)\\
&=\frac{1}{2}(1+z)\left(z+\overline w\right)^{-*}(1+\overline w)=\frac{1}{2}(1+z)k_{H^2(\HH^+)}(w,z)(1+\overline w).\\
\end{aligned}
\end{equation*}
\end{proof}

Now we want to show that the volume form on $\p\HH^+$ obtained in \eqref{hpiu} is the natural volume form for the Hardy space $H^2(\HH^+)$.
In fact, let $f(q)=\int_{0}^{+\infty}e^{-\zeta q }F(\zeta)d\zeta\in H^2(\HH^+)$. For any $I\in \s$, we can decompose the function $F$ as $F(\zeta)=F_1(\zeta)+F_2(\zeta)J$ where $J$ is an imaginary unit orthogonal to $I$ and $F_1,F_2:\rr\to L_I$. It is possible to prove (see \cite{semispazio}) that functions in $H^2(\HH^+)$ have limit at the boundary for almost any point $yI\in \p\HH^+ =\{vJ \ | \ v>0,  J\in \s \}$. If we denote by $dA_{\s}$ the usual surface element of the unit two-dimensional sphere $\s$, thanks to equation \eqref{hpiu} and to the orthogonality of $I$ and $J$, we can write  
\begin{equation*}
\begin{aligned}
&\int_{\p \HH^+}|f(yI)|^2 dVol_{\p \HH^+}(yI)=\int_{\p \HH^+}|f(yI)|^2 \frac{dVol_{Euc}(yI)}{8y^2}\\
&=\int_{0}^{+\infty}\left(\int_{\s}\left|\int_{0}^{+\infty}e^{-\zeta y I}F(\zeta)d\zeta\right|^2\frac{y^2dA_{\s}(I)}{8y^2}\right)dy\\
&=\frac{1}{8}\int_{\s}\left(\int_{0}^{+\infty}\left|\int_{0}^{+\infty}e^{-\zeta y I}\left(F_1(\zeta)+F_2(\zeta)J\right)d\zeta\right|^2dy\right)dA_{\s}(I)\\
&=\frac{1}{8}\int_{\s}\left(\int_{0}^{+\infty}\left|\int_{0}^{+\infty}e^{-\zeta y I}F_1(\zeta)d\zeta+\int_{0}^{+\infty}e^{-\zeta y I}F_2(\zeta)Jd\zeta\right|^2dy\right)dA_{\s}(I)\\
&=\frac{1}{8}\int_{\s}\left(\int_{0}^{+\infty}\left|\int_{0}^{+\infty}e^{-\zeta y I}F_1(\zeta)d\zeta\right|^2dy+\int_{0}^{+\infty}\left|\int_{0}^{+\infty}e^{-\zeta y I}F_2(\zeta)d\zeta\right|^2dy\right)dA_{\s}(I)\\
&=\frac{2\pi}{8}\int_{\s}\left(\int_{0}^{+\infty}\left|F_1(\zeta)\right|^2d\zeta+\int_{0}^{+\infty}\left|F_2(\zeta)\right)|^2d\zeta\right)dA_{\s}(I)\\
\end{aligned}
\end{equation*}
where the last equality is due to the classical Plancherel Theorem.
Therefore, thanks again to the orthogonality of $I$ and $J$, 
\[\int_{\p \HH^+}|f(yI)|^2 dVol_{\p \HH^+}(yI)=\frac{\pi}{4}\int_{\s}\left(\int_{0}^{+\infty}\left|F(\zeta)\right|^2d\zeta\right)dA_{\s}(I)=
\pi^2||f||^2_{H^2(\HH^+)}.\]
\subsection{A bilateral estimate for the distance and an application to inner functions}\label{dentro}
In the right half space model it is easier to prove a bilateral estimate for the distance associated with the invariant metric $h$.
Fix a imaginary unit $I_0$ and define the projection
\begin{equation}\label{dozza}
 \pi:x+yI\mapsto  x+yI_0,
\end{equation}
with $x$ real and $y\ge0$. Let $d_{hyp}$ be hyperbolic distance in $\HH_{I_0}^+=\{x+yI_0:\ x>0,\ y\in\RR\}$: $d_{hyp}$ is the distance associated with the Riemannian metric tensor
$ds_{hyp}^2=(dx^2+dy^2)/(4x^2)$. Let now $d_\SS$ be the usual spherical distance on the unit two-dimensional sphere $\SS$, associated with the metric tensor $ds^2_{\s}$. Then, the metric tensor associated with $h$ can be decomposed as
\[ds_h^2
=ds^2_{hyp}+\frac{y^2}{4(x^2+y^2)}ds^2_{\s}.\]
\begin{teo}\label{bilaterale}
 Let $q_j=x_j+y_jI_j$, $j=1,2$, be points in $\HH^+$: $x_j>0$, $y_j\ge0$. The following estimate for the distance function $d_h$ associated with the metric $h$ holds:
 \begin{equation*} 
 d_h(q_1,q_2)\approx d_{hyp}(\pi(q_1),\pi(q_2))+\min\left\{\frac{y_j}{|q_j|}:\ j=1,2\right\}d_\SS(I_1,I_2),
 \end{equation*}
 where $\approx$ means that we have a lower and an upper estimate for the right hand side in terms of the left hand side,
  with multiplicative constants $C_1,C_2$ independent of $q_1,q_2$.
\end{teo}
\begin{proof} We may suppose that $y_1/|q_1|\le y_2/|q_2|$. 

The upper estimate is elementary. 
Let $\gamma$ be a curve going from $q_1$ to $x_1+y_1I_2 \in \HH_{I_2}^+$ leaving $x=x_1$ and $y=y_1$ fixed, and varying the imaginary unit $I$ only. Suppose, more, 
that $I$ varies along a geodesic on $\SS$, which joins $I_1$ and $I_2$. Then,
$$
\length(\gamma)=\int_{\gamma}\frac{y_1}{2|q_1|}ds_{\s}=\frac{y_1}{2|q_1|}d_\SS(I_1,I_2).
$$
Let now $\delta$ be a hyperbolic geodesic in $L_{I_2}$, joining $x_1+y_1I_2$ and $q_2$: $\length(\delta)=d_{hyp}(\pi(q_1),\pi(q_2))$, which proves the estimate. 

The lower estimate is more delicate.
Let $\gamma$ be a curve in $\HH^+$ joining $q_1$ and $q_2$.
Then,
\begin{equation}\label{bubano}
\length(\gamma)=\int_\gamma ds_h \ge\int_{\pi(\gamma)}ds_{hyp} \ge d_{hyp}(\pi(q_1),\pi(q_2)).
\end{equation}
We have then to show that
\begin{equation}\label{mordano}
\int_\gamma ds_h \gtrsim \frac{y_1}{|q_1|}d_\SS(I_1,I_2).
\end{equation}
Since the right hand side of (\ref{mordano}) is bounded, and we have already proved (\ref{bubano}), it suffices to show that (\ref{mordano}) holds when 
$d_{hyp}(\pi(q_1),\pi(q_2))\le 1$. By elementary hyperbolic geometry, see the ``sixth model'' in \cite{cannon}, 
and using the fact that dilations $p\mapsto\lambda p$ are isometric for 
$\lambda>0$, we can assume that $\pi(q_1)$ and $\pi(q_2)$ both lie in the square $Q_n=\{x+yI_0:\ 1\le x\le2,\ n\le y\le n+1\}\subset L_{I_0}$, for some integer $n\ge0$. 
Consider now $q_3=x_3+y_3I_3$, $y_3\ge0$ the point along $\gamma$ which minimizes $y_3/|q_3|$.
We can assume that $\pi(\gamma)$ (hence, $\pi(q_3)$) is contained in  $\tilde{Q}_n=\{z=x+yI_0:\ x>0,\ y\ge0,\ 1/2\le x\le2,\ n-1/2\le y\le n+3/2\}$, 
otherwise $\length(\gamma)\ge 1$ (which would imply the estimate (\ref{mordano}) we are proving).
Let $t\ge0$ be the angle between the positive real half axis $\RR^+$ and the half line originating at $0$ and passing through $\pi(q_3)$. For $j=1,2,3$:
$$
t_j\approx\sin(t_j)=y_j/|q_j|.
$$
We have two cases. Either $y_3/|q_3|\ge 1/2\cdot y_1/|q_1|$, but then we are done because
$$
\int_\gamma \frac{y}{2|q|}ds^2_{\s} \ge \frac{y_1}{2|q_1|}d_\SS(I_1,I_2).
$$
Or $y_3/|q_3|\le 1/2\cdot y_1/|q_1|$. Then $n=0$, and 
\begin{eqnarray*}
\int_\gamma ds_h&\ge&\length(\pi(\gamma))\gtrsim \max(|\pi(z_1)-\pi(z_3)|,|\pi(z_2)-\pi(z_3)|)\ge |\pi(z_1)-\pi(z_3)|\crcr
&\gtrsim& y_1 \gtrsim\frac{y_1}{|z_1|}d_\SS(I_1,I_2).
\end{eqnarray*}
Overall, $\int_\gamma ds_h\gtrsim \frac{y_1}{|z_1|}d_\SS(I_1,I_2)$, as wished.
\end{proof}

Changing coordinates from the right half plane to the ball, we have the same bilateral estimate in the ball model.
\begin{coro}\label{bilbal}
Let $d_g$ be the invariant distance associated with the metric $g$ in the ball model and let $q_1,q_2$ be points of $\BB$. If $\pi$ is defined as in (\ref{dozza}), then:
  \begin{equation*} 
 d_g(q_1,q_2)\approx d_{hyp}(\pi(q_1),\pi(q_2))+\min\left\{\frac{y_j}{|1-q_j^2|}:\ j=1,2\right\}d_\SS(I_1,I_2).
 \end{equation*}
\end{coro}
A regular function $f:\BB\to\HH$ is \it inner \rm if (i) it maps $\BB$ into $\BB$; (ii) the limit as $r\to1$ of $f$ along the radius $r\mapsto ru$  exists for 
$a.e.$ $u$ in $\partial\BB$ and it has unitary norm.  
\begin{teo}\label{castelguelfo} Let $f:\BB\to\BB$ be an inner function. Then, $f$ is Lipschitz with respect to the metric $g$
 if and only if it is slice preserving. In this case, it is a contraction.
\end{teo}
 It is well known (see \cite{hardy}) that regular, bounded functions have radial limits along almost all radii,
 $$
f(e^{tI}):= \lim_{r\to1}f(re^{tI})
 $$
 exists for $a.e.$ $(t,I)\in[0,\pi]\times\SS$. 
 We start with a Lemma which might have independent interest; for instance, it provides a different route to  prove the classification of the isometries for the metric $g$.
\begin{lem}\label{movesil}
If $\varphi:\BB\to\BB$ is Lipschitz with respect to the metric $g$ and  
\begin{equation}\label{caterina}
\lim_{r\to1}\varphi(re^{tI_1})=e^{sJ_1 }\in\partial\BB
\end{equation}
exists, with $s\in[0,\pi]$ and $J_1\in\SS$; then for each $I_2$ in $\SS$, if the limit $\lim_{r\to1}\varphi(re^{I_2 t})$ exists, then 
\begin{equation}\label{sforza}
\lim_{r\to1}\varphi(re^{tI_2 })=e^{sJ_2 }
\end{equation}
for some  $J_2$ in $\SS$. The values of $t$ and $s$ in (\ref{sforza}) are the same as in (\ref{caterina}).
\end{lem}
\begin{proof}
Let $u_j=e^{tI_j}$, with the same $t\in[0,\pi]$. By Lipschitz continuity,
\begin{eqnarray}\label{emanuella}
 d(\varphi(ru_1),\varphi(ru_2))&\lesssim& d(ru_1,ru_2)\crcr
 &\approx&\frac{rt|I_1-I_2|}{(1-r)+rt}\lesssim|I_1-I_2|\crcr
 &\le&1.
\end{eqnarray}
By the lower estimate in Corollary \ref{bilbal} and (\ref{emanuella}), 
$$
d_{hyp}\left(\pi\left(\varphi(ru_1)\right),\pi\left(\varphi(ru_2)\right)\right)\lesssim1.
$$
But this and elementary hyperbolic geometry imply that, if $\lim_{r\to1}\varphi(ru_1)=e^{sJ_1}$, then the limit $\lim_{r\to1}\pi\left(\varphi(ru_2)\right)=L$ exists and $L=e^{sI_0}$
(recall that $\pi:\BB\to L_{I_0}$). Since $\lim_{r\to1}\varphi(re^{t I_2 })=L$ exists by hypothesis and $\pi$ is continuous, it must be $\pi(L)=e^{sI_0}$, then
$L=e^{sJ_2}$ for some $J_2$ in $\SS$.

\end{proof}


The statement of Lemma \ref{movesil}  can be sharpened in several ways. For instance, the Lipschitz assumption might be weakened to a sub-exponential growth assumption.

 We proceed with the proof of Theorem \ref{castelguelfo}.
\begin{proof}[Proof of Theorem \ref{castelguelfo}] Being inner, $f$ has boundary limits along radii $r\mapsto re^{tI}$ for $a.e.$ $I$ in $\SS$ and $t$ in $[0,\pi]$. We write
for such couples of $(t,I)$: $f(e^{tI}):=\lim_{r\to1}f(re^{tI})$. 
We can assume without loss of generality that the limit exists for two antipodal imaginary units $L$ and $-L$, and hence, in view of the Representation Formula \ref{RF}, for any $L\in\s$.
If $f$ is regular and Lipschitz with respect to the distance $d_g$, thanks on the one hand to the Representation Formula \ref{RF}, on the other hand to Lemma \ref{movesil}, we have that for any $I\in\s$
\begin{equation*}
 b(t)+Ic(t)=f(e^{tI})=e^{s(t)J(s,I)}
\end{equation*}
where 
$b(t),c(t)\in \HH$, and $s(t)\in [0,2\pi]$ and $J(t,I)\in\s$.
Then $\RRe(f(e^{tI}))=\RRe(f(e^{tL}))$ for any $L\in\s$ and in particular for $L=-I$, which gives
\[\RRe(b(t))-\langle I, c(t)\rangle= \RRe(b(t)+Ic(t))=\RRe(b(t)-Ic(t))=\RRe(b(t))+\langle I, c(t)\rangle\]
(where $\langle \cdot, \cdot \rangle$ denotes the standard scalar product in $\rr^4$).
Since $I$ is any imaginary unit, we necessarily have that $c(t)\in \rr$.

Also, comparing imaginary parts, for any $L_1,L_2\in\s$ we have $|\IIm(f(e^{tL_1}))|=|\IIm(f(e^{tL_2}))|$. Then, if $b=b_0+b_1 K$ with $b_0,b_1 \in \rr$, $K\in\s$  (omitting the dependence on $t$),  when $L_1=K$ and $L_2=-K$ we get
\[|b_1+c|=|\IIm(b_0+b_1K+cK)|=|\IIm(f(e^{tK}))|=|\IIm(f(e^{-tK}))|=|\IIm(b_0+b_1K-cK)|=|b_1-c|.\]
Therefore almost every $t\in [0,\pi]$ belongs to $D\cup E$:
$$
D=\{t:\ c(t)=0\},\ E=\{t:\ b_1(t)=0\}.
$$
Consider first the case when $t\in D$ holds $a.e.$. Then $f(e^{tI})=b(t)$ for almost every $t$. Since boundary values uniquely identify $f$ (see \cite{hardy}) 
and by invariance under rotations of $\SS$,
we deduce that 
$$
f(re^{tI})=\Phi(r,t),
$$
for some function $\Phi$. In particular, $f$ can not be open ($dim(f(\BB))\le2$), hence (see Theorem 7.4 in \cite{libroGSS}) it must be constant; thus it is not inner.
Then $E$ has positive measure. For $t$ in $E$, 
$$
 b(t)+Ic(t)=f(e^{tI})=e^{s(t)J(t,I)}
$$
with $b$ and $c$ real valued, hence $J=I$:
\begin{equation}\label{angelica}
f(e^{tI})=e^{s(t)I}
\end{equation}
for $t$ in $F$. 
By the Splitting Lemma \ref{split}, if $J\perp I$ is fixed in $\SS$, then there are holomorphic functions $F,G$ on $\B_I$ such that
$$
f(re^{\tau I})=F(re^{\tau I})+G(re^{\tau I})J.
$$
By (\ref{angelica}), $G(e^{tI})=0$ for $t$ in $E$. Since $E$ has positive measure, this implies that $G$ vanishes identically and hence $f(re^{\tau I})=F(re^{\tau I })$
for all $0\le r<1$ and $0\le \tau \le \pi$. That is, $f$ is slice preserving.

We have to verify that $f$ is a contraction with respect to the metric $g$, and this can be verified at the infinitesimal level. Let $q$ be a point in a fixed slice $\BB\cap L_I$. 
(i) Since $f$ is slice preserving, its restriction to $\BB\cap L_I$ is an inner function in the one dimensional sense, hence it is a contraction of the 
Poincar\'e-hyperbolic metric on $\BB\cap L_I$. (ii) On the other hand, preserving the slices, $f$ acts isometrically in the $\SS$ variables, with respect to the spherical metric on $\SS$. 
(iii) Now, the space tangent to $\BB\cap L_I$ at $q$ and the space tangent to $\RRe q+\SS$ at $q$ form an orthogonal decomposition, with respect to the metric $g$, 
of the space tangent to $\BB$. From the expression for $g$ given in (\ref{metricapolare}) and facts (i)-(iii) one easily deduces that $g$ is a contraction.
\end{proof}

{Nicola Arcozzi\\ 
\normalsize Dipartimento di Matematica, Universit\`a di Bologna \\ 
\normalsize Piazza di Porta San Donato 5, 40126 Bologna, Italy,  nicola.arcozzi@unibo.it \\
\and Giulia Sarfatti \\ 
\normalsize Dipartimento di Matematica, Universit\`a di Bologna \\ 
\normalsize Piazza di Porta San Donato 5, 40126 Bologna, Italy,  giulia.sarfatti@unibo.it}

\end{document}